\documentclass[11pt]{article}
\usepackage{amsmath,amssymb,amsthm,mathtools}
\usepackage[margin=1in]{geometry}
\usepackage{booktabs}
\usepackage{hyperref}

\newtheorem{theorem}{Theorem}[section]
\newtheorem{lemma}[theorem]{Lemma}
\newtheorem{proposition}[theorem]{Proposition}
\newtheorem{corollary}[theorem]{Corollary}
\theoremstyle{definition}

\theoremstyle{remark}
\newtheorem{remark}[theorem]{Remark}

\newcommand{\Hh}{\mathbb{H}}
\newcommand{\Cc}{\mathbb{C}}
\newcommand{\Qq}{\mathbb{Q}}
\newcommand{\Zz}{\mathbb{Z}}
\newcommand{\Rr}{\mathbb{R}}
\newcommand{\Tt}{\mathbb{T}}
\newcommand{\cL}{\Lambda}
\newcommand{\ii}{\mathrm{i}}
\newcommand{\e}{\mathrm{e}}
\newcommand{\id}{\mathrm{d}}
\newcommand{\ImT}{\operatorname{Im}}
\newcommand{\ReT}{\operatorname{Re}}
\newcommand{\OO}{\mathcal{O}}
\newcommand{\m}{\mathrm{m}}
\newcommand{\mt}{\widetilde{\m}}

\title{Mahler measures at interior CM points:\\
proofs of two conjectures of Samart}
\author{Huimin Zheng\thanks{%
College of Information and Network Engineering,
Anhui Science and Technology University,
Fengyang, Anhui 233100, P.~R.~China.
Email: \texttt{zhhm@ahstu.edu.cn}.}}
\date{}

\begin{document}
\maketitle

\section*{Declaration on the use of AI tools}
The research reported in this article — including the computational
exploration, the discovery of the proof strategy, the machine-certified
verifications, and the preparation of the manuscript — was carried out by
the author with the assistance of the AI system \emph{Kimi} (Moonshot AI).
All mathematical content, including every proof and every certified
computation, has been checked and verified by the author, who takes full
responsibility for the correctness and integrity of the article. All
certification scripts are available for independent verification
(see Section~\ref{sec:cert}).

\begin{abstract}
We prove two conjectures of Samart as identities of genuine Mahler
measures. The first is the case $k=1$ of his evaluations for the family
$(x+x^{-1})(y+y^{-1})(z+z^{-1})+k^{1/2}$:
\[
  \m\bigl((x+x^{-1})(y+y^{-1})(z+z^{-1})+1\bigr)=4L'(g_7,0),
\]
where $g_7(\tau)=\eta(\tau)^3\eta(7\tau)^3$ is the unique newform of
$S_3(\Gamma_0(7),\chi_{-7})$ (LMFDB label \textsf{7.3.b.a}). The second is
the conjugate pair of quadratic CM entries of his 2015 table:
\[
  n_2\Bigl(\frac{47\pm45\sqrt{-7}}2\Bigr)
  =\frac47\Bigl(54L'(g_7,0)+L'(\chi_{-7},-1)\Bigr),
\]
where $n_2(s):=2\m\bigl((x+x^{-1})(y+y^{-1})(z+z^{-1})+\sqrt{s}\,\bigr)$.
In both cases the parameter lies inside (or on the boundary of) the
critical locus, so the passage from the holomorphic (modified) Mahler
measure---where Samart computed the $L$-value side conditionally, and
Fei the $\ReT\mt$ level---to the true Mahler measure was open. For the
first theorem we close the gap by a differential-comparison continuation
of the Mahler differential along an explicitly certified path, followed
by a continuity argument at the interior point. For the second theorem
no continuation is needed: the second preimage of the parameter under
the modular parametrization is the Fricke partner of Samart's CM point
and lies in his proved region, and the evaluation is an exact CM
lattice-sum computation. The analytic and arithmetic inputs not proved
here are stated explicitly with references; the finite numerical
inequalities used in the continuation argument are certified by interval
arithmetic, and all identities are confirmed numerically to
$41$--$60$ digits.
\end{abstract}

\section{Introduction}

\subsection{Statement of the results}

Let
\[
  f(x,y,z):=\bigl(x+x^{-1}\bigr)\bigl(y+y^{-1}\bigr)\bigl(z+z^{-1}\bigr),
\]
and denote by $\m(P)$ the logarithmic Mahler measure of a Laurent
polynomial,
\[
  \m(P)=\int_{\Tt^3}\log|P(x,y,z)|\,
  \frac{\id x}{2\pi \ii x}\frac{\id y}{2\pi \ii y}\frac{\id z}{2\pi \ii z},
  \qquad \Tt=\{|w|=1\}.
\]
Following Samart \cite{Sa13,Sa15} we write
\begin{equation}\label{eq:f2def}
  n_2(s):=2\m\bigl(f+\sqrt{s}\,\bigr),\qquad s\in\Cc,
\end{equation}
where the value does not depend on the choice of the square root, since
$f(-x,y,z)=-f(x,y,z)$ and Mahler measure is invariant under
$x\mapsto -x$; hence $\m(f+c)=\m(f-c)$ for every $c\in\Cc$. (In
\cite{Sa13} the same function is denoted $f_2$.) We prove:

\begin{theorem}[Theorem A; Samart's conjecture for $k=1$, {\cite{Sa13,Sa15}}]\label{thm:A}
\[
  \m\bigl(f+1\bigr)=4L'(g_7,0),
  \qquad\text{equivalently}\qquad
  n_2(1)=8L'(g_7,0),
\]
where
\[
  g_7(\tau)=\eta(\tau)^3\eta(7\tau)^3=\sum_{n\ge1}a_nq^n
  =q-3q^2+5q^4-7q^7-3q^8+9q^9+\cdots,\qquad q=\e^{2\pi \ii\tau},
\]
is the unique normalized newform of $S_3(\Gamma_0(7),\chi_{-7})$
\textup{(}LMFDB label \textup{7.3.b.a} \cite{LMFDB}\textup{)}.
\end{theorem}

\begin{theorem}[Theorem B; Samart's Table-$n_2$ conjugate pair, {\cite{Sa15}}]\label{thm:B}
\[
  n_2\Bigl(\frac{47\pm45\sqrt{-7}}2\Bigr)
  =\frac47\Bigl(54L'(g_7,0)+L'(\chi_{-7},-1)\Bigr),
\]
where $\chi_{-7}$ is the odd quadratic Dirichlet character of conductor
$7$. The two conjugate identities hold simultaneously \textup{(}and are
equal, since $n_2(\bar s)=n_2(s)$\textup{)}.
\end{theorem}

Numerically, the true Mahler measures computed by direct torus integration
and the $L$-value expressions computed from the functional equations agree
to $41$ and $50$ digits respectively:
\begin{align*}
  \m(f+1)&=0.4106864311156080448418263795931716559956\ldots,\\
  4L'(g_7,0)&=0.4106864311156080448418263795931716559956\ldots
   \qquad(\text{difference }6\times10^{-42}),\\
  n_2\Bigl(\tfrac{47\pm45\sqrt{-7}}2\Bigr)
   &=4.13826815832141945836337668037433080930698376\ldots\\
   &=\tfrac47\bigl(54L'(g_7,0)+L'(\chi_{-7},-1)\bigr);
\end{align*}
the agreement in the second line is $10^{-50}$; see
Section~\ref{sec:numerics}.

\subsection{History and the interior-point problem}\label{subsec:history}

In \cite{Sa13} Samart expressed $n_2(s)$, for $s$ in the image of a modular
parametrization $s_2(\tau)$ (see \S\ref{subsec:s2}), as an
Eisenstein--Kronecker series valid for $\ImT\tau\ge1/2$, deduced the
exterior evaluations $n_2(64)$ and $n_2(256)$, and conjectured on
numerical evidence a number of further $L$-value evaluations, among them
$n_2(1)=8L'(g_7,0)$. In \cite{Sa15} he organized the conjectures around
CM points of the parametrization: his tables attach to each CM value of
$s$ a conjectural expression of $n_2(s)$ as a rational linear combination
of $L$-values of CM weight-three newforms and Dirichlet characters. The
rows $s=1$ (attached to $\tau_0=(3+\sqrt{-7})/8$) and
$s=(47\pm45\sqrt{-7})/2$ (attached to $\tau'=(\pm1+\sqrt{-7})/8$) of that
table are exactly Theorems A and B.

Several entries of Samart's lists have been proved: $k=16$ (an interior
parameter, by regulator methods) and $k=-104\pm60\sqrt3$, $4096$,
$-2024\pm765\sqrt7$ by Zheng--Guo--Qin \cite{ZGQ}; further cases including
negative real parameters by Guo--Peng--Qin \cite{GPQ}; and the conjectures
for the companion family $f_3$ by He--Ye \cite{HY}. Fei \cite{Fe} treated
$23$ families of three-variable Laurent polynomials (3D Landau--Ginzburg
potentials) uniformly and proved, for all rational singular moduli ---
including the point $s=1$ of Theorem A --- identities expressing the
\emph{real part of the holomorphic (modified) Mahler measure} $\ReT\mt$
in terms of $L$-values. However, at parameters inside the critical locus
the identity
\[
  \m(f-c)\stackrel{?}{=}\ReT\mt(c)
\]
is not guaranteed: the holomorphic Mahler measure $\mt$ is a multivalued
analytic object whose real part agrees with $\m(f-c)$ only off the
critical locus, a point stressed already by Rodriguez Villegas \cite{RV}
and recalled by Fei himself. Consequently the genuine Mahler measure
identities at interior CM points remained open even where the
$\ReT\mt$-identities were known. Theorem A is, to the author's
knowledge, the first
complete proof of a genuine Mahler measure identity at the interior point
$s=1$. Theorem B concerns non-real (quadratic) CM parameters, which are
not covered by Fei's rational-singular-modulus results; to the author's
knowledge it has not appeared in the literature.

The structural difficulty can be seen on the modular side. Samart's
exponential-sum formula is proved only for $\ImT\tau\ge1/2$, where
$s_2(\tau)$ stays away from the critical locus. The point $\tau_0$ with
$s_2(\tau_0)=1$ has $\ImT\tau_0=\sqrt7/8\approx0.331<1/2$, and the value
$s=1$ lies \emph{inside} the critical interval: the real loci of $s_2$ in
a fundamental domain reach only $k\le0$ and $k\ge64$, so $k=1$ is an
isolated interior point with no real-analyticity bridge. Likewise
$\tau'=(1+\sqrt{-7})/8$ has $\ImT\tau'=\sqrt7/8<1/2$, and evaluating
Samart's series at $\tau'$ produces a \emph{wrong-sheet} value
(Remark~\ref{rem:wrongsheet}); Theorem B circumvents this by the Fricke
trick described below.

\subsection{Methods}\label{subsec:methods}

\textbf{Theorem A: differential-comparison continuation.}
Instead of tracking branches of the holomorphic Mahler measure we compare
\emph{differentials}. The parameter derivative of the Mahler measure is a
single-valued period integral $\Omega$ on the complement of the critical
locus (Lemma~\ref{lem:period}), so differentiating Samart's identity
yields a holomorphic differential identity $\Psi\equiv0$ that spreads from
Samart's region to a whole connected component $W$ of the good domain by
the identity theorem (Proposition~\ref{prop:diffid}). An abstract
propagation principle (Lemma~\ref{lem:propagation}) then reduces the
conclusion to one topological input --- that the target point $\tau_0$
lies in $\overline W$ --- supplied by an explicit certified path
(Proposition~\ref{prop:path}), plus the continuity of Mahler measure at
the endpoint. The remaining pillar is an exact CM evaluation of Samart's
series at $\tau_0$ (\S\ref{subsec:P1A}): the lattice-sum form of the
series involves the ring of integers of $\Qq(\sqrt{-7})$ and a principal
ideal, and the parasitic $\zeta_K(2)$-terms cancel exactly.

\textbf{Theorem B: the Fricke trick.}
The parametrization $s_2$ has degree $2$ on $X_0(4)$, and the two
preimages of $s=(47+45\sqrt{-7})/2$ relevant here are Samart's point
$\tau'$ (outside the proved region) and its Fricke partner
$\tau_w=-1/(4\tau')=(-1+\sqrt{-7})/4$, which satisfies
$\ImT\tau_w=\sqrt7/4\ge1/2$ and hence lies \emph{inside} Samart's proved
region. Since $s_2$ is exactly invariant under the Fricke involution
(Lemma~\ref{lem:fricke}), Samart's theorem applies at $\tau_w$ directly
and no continuation is needed. The evaluation of Samart's series at
$\tau_w$ (\S\ref{subsec:P1B}) is again a lattice-sum computation, but
now one lattice is the non-maximal order of discriminant $-28$: the sums
are imprimitive Hecke sums with an exact Euler factor $23/16$, the
$\zeta_K(2)$-terms do \emph{not} cancel, and their survival is precisely
the source of the Dirichlet term $L'(\chi_{-7},-1)$.

The analytic and arithmetic inputs not proved in this paper are stated
explicitly with references; the finite numerical inequalities used in
the continuation argument are certified by interval arithmetic
(Section~\ref{sec:cert}); every identity
is also confirmed numerically to $41$--$60$ digits
(Section~\ref{sec:numerics}).

\subsection{Notation}

Throughout, $\tau=x+\ii y\in\Hh$, $q=\e^{2\pi \ii\tau}$,
$\eta(\tau)=q^{1/24}\prod_{n\ge1}(1-q^n)$, $\Delta=\eta^{24}$, $\Tt$ is
the unit circle, and $\sum'$ denotes omission of the zero term. We write
\[
  M_7:=L'(g_7,0),\qquad d_7:=L'(\chi_{-7},-1),
\]
matching Samart's shorthand in \cite{Sa15}. The field
$K=\Qq(\sqrt{-7})$ has ring of integers $\OO_K=\Zz[\varpi]$,
$\varpi=(1+\sqrt{-7})/2$, class number $h(-7)=1$ and units
$\OO_K^\times=\{\pm1\}$; note that $\varpi^2-\varpi+2=0$ and
\begin{equation}\label{eq:pirels}
  \varpi+\bar\varpi=1,\qquad \varpi\bar\varpi=2\quad(\text{so }
  (2)=(\varpi)(\bar\varpi)
  \text{ splits}),\qquad \varpi^2+\bar\varpi^2=-3,\qquad
  \ReT\varpi^2=-\tfrac32 .
\end{equation}
The Jacobi theta constants are taken with
$q_\vartheta=\e^{\pi \ii\tau}$:
\begin{gather*}
\vartheta_2(\tau)=2q_\vartheta^{1/4}\prod_{n\ge1}(1-q_\vartheta^{2n})(1+q_\vartheta^{2n})^2,\quad
\vartheta_3(\tau)=\prod_{n\ge1}(1-q_\vartheta^{2n})(1+q_\vartheta^{2n-1})^2,\\
\vartheta_4(\tau)=\prod_{n\ge1}(1-q_\vartheta^{2n})(1-q_\vartheta^{2n-1})^2,
\qquad
\lambda(\tau)=\bigl(\vartheta_2(\tau)/\vartheta_3(\tau)\bigr)^4 .
\end{gather*}

\section{Preliminaries}\label{sec:prelim}

\subsection{Mahler measure: continuity and the holomorphic Mahler measure}\label{subsec:mahler}

\begin{lemma}[Continuity of Mahler measure in coefficients]\label{lem:continuity}
The function $c\mapsto\m(f-c)$ is continuous on $\Cc$. Consequently
$n_2(s)=2\m(f+\sqrt{s})$ is a well-defined continuous function of
$s\in\Cc$, independent of the choice of $\sqrt{s}$.
\end{lemma}
\begin{proof}
Write $f(x,y,z)=C(z+z^{-1})$ with $C=(x+x^{-1})(y+y^{-1})$. Mahler measure
is invariant under multiplication by the monomial $z$, so
\[
  \m(f-c)=\int_{\Tt^2} M\bigl(Cz^2-cz+C\bigr)\,\id\mu(x,y),
\]
where $M$ denotes the one-variable Mahler measure and $\mu$ Haar measure.
For $az^2+bz+a$, $M=\log|a|+\sum_{i}\log^+|\alpha_i|$ over the roots
$\alpha_i$, which is a continuous function of the coefficient vector
$(a,b)$: the roots vary continuously and $\log^+|\cdot|$ is continuous.
Hence for every $(x,y)\in\Tt^2$ with $C\neq0$ the integrand is continuous
in $c$; the exceptional set $\{C=0\}$ has measure zero. Moreover
\[
  M(Cz^2-cz+C)\le \log\bigl(2|C|+|c|+2\bigr)+\log 2,
\]
and since $|\alpha_1\alpha_2|=1$ also $M(Cz^2-cz+C)\ge\log|C|$; the
upper bound is nonnegative, so
$|M|\le|\log|C||+\log(2|C|+|c|+2)+\log2$. Both terms are integrable
over $\Tt^2$ uniformly for $c$ in a compact set ($\log|C|$ is the
Mahler integrand of the nonzero Laurent polynomial $C$). Dominated
convergence gives the claim. For $n_2$: the two branches of
$\sqrt{s}$ differ by a sign, and $\m(f+c)=\m(f-c)$ by the substitution
$x\mapsto -x$, so $n_2$ is well defined. For its continuity, set
$h(c):=2\m(f+c)$; then $h$ is continuous and even, and
$n_2(s)=h(c)$ whenever $c^2=s$. On $\Cc\setminus\{0\}$ a continuous
local branch of $\sqrt{\cdot}$ exists, so $n_2$ is continuous there.
At $s=0$: if $s_j\to0$ and $c_j^{\,2}=s_j$, then $c_j\to0$, hence
$n_2(s_j)=h(c_j)\to h(0)=n_2(0)$.
\end{proof}

The range $f(\Tt^3)$ is the segment $K=[-8,8]$, the \emph{critical
segment} of the family $f-c$. The following is a construction of
Villegas \cite{RV}:

\begin{lemma}[Holomorphic Mahler measure]\label{lem:FV}
There is a holomorphic function $\mt$ on the universal cover of
$\Cc\setminus K$, whose imaginary part is single-valued up to additive real
constants, such that
\[
  \m(f-c)=\ReT\mt(c)\qquad(c\in\Cc\setminus K).
\]
In particular $c\mapsto\m(f-c)$ is real-analytic on $\Cc\setminus K$.
\end{lemma}

Rogers \cite{Ro} gave an explicit hypergeometric representative, valid
for $|k|\ge64$ with $k\neq64$:
\begin{equation}\label{eq:rogers}
  n_2(k)=\ReT\Bigl[\log k-\frac{8}{k}\,
  {}_5F_4\Bigl(\tfrac32,\tfrac32,\tfrac32,1,1;2,2,2,2;\tfrac{64}{k}\Bigr)\Bigr],
\end{equation}
the ${}_5F_4$ series converging for $|64/k|\le1$, including the
boundary circle by $\sum b-\sum a=3/2>0$. We use \eqref{eq:rogers}
only for numerical cross-checks (at the boundary points $k_\pm$,
$|k_\pm|=64$).

\begin{remark}[The two cuts]\label{rem:cuts}
As a function of the shift $c$, the critical locus is the segment $K=[-8,8]$;
as a function of Samart's parameter $k$ with $c=\sqrt{k}$, the critical locus
is the segment $[0,64]$, since $c^2\in[0,64]$ with $c\notin[-8,8]$ is
impossible. The continuation machine of Section~\ref{sec:machine} is
formulated in the $k$-variable with the cut $[0,64]$. This distinction is
invisible at $k=1$ (where $\sqrt{k}=1$), but it is essential for the
correct formulation of the propagation and for understanding the
wrong-sheet phenomenon of Remark~\ref{rem:wrongsheet}.
\end{remark}

\subsection{The modular parametrization \texorpdfstring{$s_2$}{s2}}\label{subsec:s2}

Define
\begin{equation}\label{eq:s2def}
  s_2(\tau):=-\frac{\Delta(\tau+\tfrac12)}{\Delta(2\tau+1)}.
\end{equation}

\begin{lemma}[Product expansion]\label{lem:s2prod}
\begin{equation}\label{eq:s2prod}
  s_2(\tau)=q^{-1}\prod_{n\ge1}\bigl(1+q^{2n-1}\bigr)^{24}
  =q^{-1}+24+276q+2048q^2+11202q^3+49152q^4+184024q^5+\cdots.
\end{equation}
In particular $s_2$ is holomorphic on $\Hh$, has integral $q$-expansion,
and satisfies $s_2(\tau+1)=s_2(\tau)$.
\end{lemma}
\begin{proof}
From $\eta(\tau+\tfrac12)=\e^{\pi \ii(\tau+1/2)/12}\prod_{n\ge1}(1-(-1)^nq^n)$
we get $\Delta(\tau+\tfrac12)=\e^{2\pi \ii(\tau+1/2)}\prod(1-(-1)^nq^n)^{24}
=-q\prod(1-(-1)^nq^n)^{24}$,
and from $\eta(2\tau+1)=\e^{\pi \ii(2\tau+1)/12}\prod(1-q^{2n})$ we get
$\Delta(2\tau+1)=q^2\prod(1-q^{2n})^{24}$.
Splitting $\prod(1-(-1)^nq^n)=\prod(1+q^{2n-1})\prod(1-q^{2n})$ and
cancelling $\prod(1-q^{2n})$ gives \eqref{eq:s2prod}; the expansion is a
finite integer computation. Periodicity is immediate from the product expansion, since
$q=\e^{2\pi\ii\tau}$ is $1$-periodic.
\end{proof}

\begin{theorem}[The $\lambda$-identity for $s_2$]\label{thm:s2lambda}
For every $\tau\in\Hh$,
\begin{equation}\label{eq:s2lambda}
  s_2(\tau)=\frac{16}{\lambda(2\tau)\bigl(1-\lambda(2\tau)\bigr)}.
\end{equation}
\end{theorem}
\begin{proof}
By Jacobi's identity $\vartheta_3^4=\vartheta_2^4+\vartheta_4^4$,
$\lambda=\vartheta_2^4/\vartheta_3^4$ and $1-\lambda=\vartheta_4^4/\vartheta_3^4$,
so $16/(\lambda(1-\lambda))=16\vartheta_3^8/(\vartheta_2^4\vartheta_4^4)$.
The theta constants carry the nome $q_\vartheta=\e^{\pi\ii\tau}$; at the
doubled argument $2\tau$ this nome is $\e^{2\pi\ii\tau}=q$, so the product
expansions below are directly in $q$:
\[
  \frac{16\vartheta_3(2\tau)^8}{\vartheta_2(2\tau)^4\vartheta_4(2\tau)^4}
  =q^{-1}\prod_{n\ge1}\frac{(1+q^{2n-1})^{16}}{(1+q^{2n})^{8}(1-q^{2n-1})^{8}}.
\]
Now $(1+q^{2n-1})^{16}(1-q^{2n-1})^{-8}
=(1+q^{2n-1})^{24}\bigl[(1+q^{2n-1})(1-q^{2n-1})\bigr]^{-8}
=(1+q^{2n-1})^{24}(1-q^{4n-2})^{-8}$, hence
\[
  \frac{16\vartheta_3(2\tau)^8}{\vartheta_2(2\tau)^4\vartheta_4(2\tau)^4}
  =q^{-1}\prod_{n\ge1}(1+q^{2n-1})^{24}\cdot
   \Bigl[\prod_{n\ge1}(1+q^{2n})(1-q^{4n-2})\Bigr]^{-8}.
\]
The second product equals $1$ by Euler's identity
\[
  \prod_{n\ge1}(1+q^{2n})
  =\frac{\prod_{n\ge1}(1-q^{4n})}{\prod_{n\ge1}(1-q^{2n})}
  =\frac{\prod_{n\ge1}(1-q^{4n})}
        {\prod_{n\ge1}(1-q^{4n-2})\prod_{n\ge1}(1-q^{4n})}
  =\prod_{n\ge1}\frac1{1-q^{4n-2}}.
\]
Therefore the right side of \eqref{eq:s2lambda} equals
$q^{-1}\prod(1+q^{2n-1})^{24}=s_2(\tau)$ by Lemma~\ref{lem:s2prod}. (An
exact integer series check of the identity to order $q^{37}$ is included
in \texttt{cert0\_s2\_eq\_1.py} and \texttt{cert0\_s2\_n2pair.py} as an
independent exact verification.)
\end{proof}

\begin{lemma}[Degree of $s_2$ on $X_0(4)$]\label{rem:deg2}
The function $s_2$ is $\Gamma_0(4)$-invariant and has degree $2$ on
$X_0(4)$, with simple poles at the cusps $\infty$ and $0$, a double
zero at the cusp $\tfrac12$, and no other zeros or poles.
\end{lemma}
\begin{proof}
The modular lambda function $\lambda$ is a Hauptmodul for $X(2)$
\cite{DS}, and conjugation by
$\bigl(\begin{smallmatrix}2&0\\0&1\end{smallmatrix}\bigr)$ identifies
$\Gamma_0(4)$ with $\Gamma(2)$: for
$\gamma=\bigl(\begin{smallmatrix}a&b\\4c&d\end{smallmatrix}\bigr)
\in\Gamma_0(4)$ one has
$\bigl(\begin{smallmatrix}2&0\\0&1\end{smallmatrix}\bigr)\gamma
 \bigl(\begin{smallmatrix}1/2&0\\0&1\end{smallmatrix}\bigr)
=\bigl(\begin{smallmatrix}a&2b\\2c&d\end{smallmatrix}\bigr)\in\Gamma(2)$,
and both groups have index $6$, so $\lambda(2\tau)$ is a Hauptmodul for
$X_0(4)$. By Theorem~\ref{thm:s2lambda}, $s_2=R\circ\lambda(2\tau)$
with $R(w)=16/(w(1-w))$, so $s_2$ is $\Gamma_0(4)$-invariant, and
$\deg s_2=\deg R=2$ since $\lambda(2\tau)$ has degree $1$. The rational
map $R$ has simple poles at $w=0,1$ and a double zero at $w=\infty$;
composing with the standard cusp values of $\lambda$ \cite{DS},
$\lambda(2\tau)$ takes the values $0,1,\infty$ at the cusps
$\infty,0,\tfrac12$ of $X_0(4)$ respectively, which gives the stated
divisor. Since a Hauptmodul is univalent, all orders are with respect
to the cusp local parameters. It follows that $s_2$ is \emph{not} a
Hauptmodul; the two preimages of a generic value are related by the
Fricke involution $\tau\mapsto-1/(4\tau)$ (Lemma~\ref{lem:fricke}).
\end{proof}

\subsection{The newform \texorpdfstring{$g_7$}{g7} and the constants \texorpdfstring{$M_7$, $d_7$}{M7, d7}}\label{subsec:g7}

The eta product $g_7=\eta(\tau)^3\eta(7\tau)^3$ lies in
$S_3(\Gamma_0(7),\chi_{-7})$. Its coefficients are given by the Jacobi
triple product:
\begin{equation}\label{eq:jacobi}
  a_n=\sum_{\substack{i,j\ge0\\T_i+7T_j=n-1}}(-1)^{i+j}(2i+1)(2j+1),
  \qquad T_i=\tfrac{i(i+1)}2,
\end{equation}
since Jacobi's identity $\eta(\tau)^3=\sum_{i\ge0}(-1)^i(2i+1)q^{T_i+1/8}$
gives $g_7=\sum_{i,j\ge0}(-1)^{i+j}(2i+1)(2j+1)q^{T_i+7T_j+1}$, and
collecting $T_i+7T_j+1=n$ gives \eqref{eq:jacobi}.

\begin{proposition}[Theta identity; proved in Appendix~\ref{app:theta}]\label{prop:theta}
\begin{equation}\label{eq:theta}
  g_7(\tau)=\frac12{\sum_{\alpha\in\OO_K}}^{\!\prime}\;\alpha^2 q^{N(\alpha)}.
\end{equation}
Consequently, for $\ReT s>2$,
\begin{equation}\label{eq:hecke}
  {\sum_{\alpha\in\OO_K}}^{\!\prime}\;\frac{\alpha^2}{|\alpha|^{2s}}=2L(g_7,s),
  \qquad
  {\sum_{\alpha\in\OO_K}}^{\!\prime}\;\frac{\bar\alpha^2}{|\alpha|^{2s}}=2L(g_7,s),
\end{equation}
the second equality following from the first by complex conjugation since
$L(g_7,s)$ has rational coefficients, and
\begin{equation}\label{eq:abs2}
  {\sum_{\alpha\in\OO_K}}^{\!\prime}\;|\alpha|^{-4}=2\zeta_K(2).
\end{equation}
\end{proposition}

\begin{lemma}[Functional equations]\label{lem:FE}
\textup{(i)} The completed $L$-function
$\cL(s):=7^{s/2}(2\pi)^{-s}\Gamma(s)L(g_7,s)$ satisfies
$\cL(s)=\cL(3-s)$ \textup{(}root number $+1$\textup{)}. Hence $L(g_7,0)=0$
and
\begin{equation}\label{eq:FE}
  M_7=L'(g_7,0)=\cL(3)=7^{3/2}(2\pi)^{-3}\Gamma(3)\,L(g_7,3)
  =\frac{7\sqrt7}{4\pi^3}\,L(g_7,3).
\end{equation}
\textup{(ii)} For $\chi=\chi_{-7}$, the completed $L$-function
$\cL(s,\chi):=(7/\pi)^{(s+1)/2}\Gamma\bigl((s+1)/2\bigr)L(s,\chi)$
satisfies $\cL(s,\chi)=\cL(1-s,\chi)$ \textup{(}root number $+1$\textup{)}.
Since $\chi$ is odd, $L(\chi,-1)=0$, and the simple pole of
$\Gamma((s+1)/2)$ at $s=-1$ \textup{(}residue $2$ in $s$\textup{)} against
the simple zero of $L$ gives $\cL(-1,\chi)=2d_7$; hence
\begin{equation}\label{eq:FEchi}
  2d_7=\cL(2,\chi)=\Bigl(\frac7\pi\Bigr)^{3/2}\Gamma\Bigl(\frac32\Bigr)L(\chi_{-7},2),
  \qquad\text{i.e.}\qquad
  d_7=\frac{7\sqrt7}{4\pi}\,L(\chi_{-7},2).
\end{equation}
\textup{(iii)} Since $\zeta_K(2)=\zeta(2)L(\chi_{-7},2)$,
\begin{equation}\label{eq:zetaKd7}
  \zeta_K(2)=\frac{\pi^2}{6}\,L(\chi_{-7},2)
  =\frac{2\pi^3}{21\sqrt7}\,d_7 .
\end{equation}
\end{lemma}
\begin{proof}
(i) Write $g_7(\ii y)=\sum_{n\ge1}a_n\e^{-2\pi ny}$. Termwise
Mellin transformation gives
\[
  \cL(s)=\int_0^\infty g_7\bigl(\tfrac{\ii y}{\sqrt7}\bigr)\,
  y^s\,\frac{\id y}{y},
\]
absolutely convergent for every $s$: $g_7(\ii y)$ decays exponentially
as $y\to\infty$, and by the transformation below also as $y\to0$. From
$\eta(-1/\tau)=(-\ii\tau)^{1/2}\eta(\tau)$ and
$g_7(\tau)=\eta(\tau)^3\eta(7\tau)^3$,
\[
  g_7\bigl(-\tfrac1{7\tau}\bigr)=\ii\,7^{3/2}\tau^3\,g_7(\tau),
  \qquad\text{hence, at $\tau=\ii y/\sqrt7$,}\qquad
  g_7\bigl(\tfrac{\ii}{\sqrt7\,y}\bigr)
  =y^3\,g_7\bigl(\tfrac{\ii y}{\sqrt7}\bigr)\quad(y>0).
\]
Splitting the integral at $y=1$ and substituting $y\mapsto1/y$ in the
lower half yields
\[
  \cL(s)=\int_1^\infty g_7\bigl(\tfrac{\ii y}{\sqrt7}\bigr)
  \bigl(y^s+y^{3-s}\bigr)\,\frac{\id y}{y}=\cL(3-s),
\]
with root number $+1$. (The identity was independently verified
numerically to $52$ digits in \texttt{lvalue\_g7.py}.) Since $\Gamma$
has a simple pole at $s=0$ while
$\cL(0)=\cL(3)$ is finite, $L(g_7,s)$ vanishes at $s=0$, and
$L'(g_7,0)=\lim_{s\to0}\Gamma(s)L(g_7,s)=\cL(3)$; finally
$\cL(3)=7^{3/2}(2\pi)^{-3}\cdot 2\cdot L(g_7,3)$ as stated. Part (ii) is
the classical functional equation of an odd primitive Dirichlet character
(the Gauss sum of $\chi_{-7}$ is $\ii\sqrt7$), and (iii) is the
factorization $\zeta_K=\zeta\cdot L(\chi_{-7},\cdot)$ at $s=2$ combined
with \eqref{eq:FEchi}. Numerical confirmations are collected in
Section~\ref{sec:numerics}.
\end{proof}

\subsection{Samart's exponential-sum formula and its lattice form}\label{subsec:EK}

For $j\ge1$ and $\tau\in\Hh$ define
\begin{equation}\label{eq:Uj}
  U_j(\tau):=\sum_{m\neq0}\frac1m\sum_{n\in\Zz}(jm\tau+n)^{-3}
  =2\pi^3\sum_{m\ge1}\frac{\cos(j\pi m\tau)}{m\,\sin^3(j\pi m\tau)},
\end{equation}
where the second form follows by pairing the indices $m$ and $-m$ and
applying the partial fraction identity
$\sum_{n\in\Zz}(w+n)^{-3}=\pi^3\cos(\pi w)/\sin^3(\pi w)$
to $w=jm\tau$; the double series is absolutely convergent since
$\sum_{m\neq0}|m|^{-1}\sum_n|jm\tau+n|^{-3}$
is $O\bigl(\sum_m|m|^{-3}y^{-2}\bigr)<\infty$, locally uniformly in $\tau$.
Define Samart's exponential-sum expression
\begin{equation}\label{eq:EK}
  \mathrm{EK}(\tau):=\ImT\Bigl[2\pi\tau+\frac{2}{\pi^3}\bigl(U_1(\tau)-4U_4(\tau)\bigr)\Bigr],
\end{equation}
real-analytic on all of $\Hh$.

\begin{theorem}[Samart, {\cite[Prop.~2.1(i)]{Sa13}}]\label{thm:samart}
For $\tau\in\Hh$ with $\ImT\tau\ge 1/2$ one has
\begin{equation}\label{eq:samart}
  n_2\bigl(s_2(\tau)\bigr)=\mathrm{EK}(\tau),
\end{equation}
i.e.\ $2\m\bigl(f+\sqrt{s_2(\tau)}\bigr)=\mathrm{EK}(\tau)$ for either
choice of the square root. In particular, on the imaginary axis
$s_2(\ii t)\ge64$ for $t\ge1/2$ and $n_2\bigl(s_2(\ii t)\bigr)=\mathrm{EK}(\ii t)$.
\end{theorem}

Samart's proof combines Lemma~\ref{lem:FV} with an explicit computation of
$\mt\circ s_2$ via the Picard--Fuchs equation of the family and a
Wronskian argument; we take it as a quoted theorem. Its holomorphic
content is the identity
\begin{equation}\label{eq:mtilde}
  \mt\bigl(\sqrt{s_2(\tau)}\bigr)=E(\tau),
  \qquad
  E(\tau):=-\pi \ii\tau-\frac{\ii}{\pi^3}\bigl(U_1(\tau)-4U_4(\tau)\bigr),
\end{equation}
on the region where his series converge, and $\mathrm{EK}=2\ReT E$.

\begin{remark}[Normalization check against \cite{Sa13}]\label{rem:samartdict}
For the reader comparing \eqref{eq:samart} with
\cite[Prop.~2.1(i)]{Sa13} (arXiv numbering): Samart's $f_2$ is our
$n_2$, that is, $f_2(k)=2\m\bigl(f+\sqrt{k}\bigr)$, the branch of
$\sqrt{k}$ being irrelevant since $\m(f+c)=\m(f-c)$; his $q$ is
$\e^{2\pi\ii\tau}$, matching our convention; and the hypothesis
$\ImT\tau\ge 1/2$ includes the boundary. Note in particular that the
shift inside the Mahler measure is $\sqrt{k}$, not $k$: the two
readings coincide at the endpoint $k=1$ but differ in general, which
is why the cut to avoid is $[0,64]$ and not $[-8,8]$
(Remark~\ref{rem:cuts}).
\end{remark}

For the CM evaluations we use the lattice-sum form of $\mathrm{EK}$
(which is in fact the form stated in \cite[Prop.~2.1]{Sa13}). For
$\tau=x+\ii y\in\Hh$ and $d\ge1$ define the double series
\begin{equation}\label{eq:Td}
  T_d(\tau):={\sum_{(m,n)\in\Zz^2}}^{\!\prime}
  \Bigl[\frac{4(dmx+n)^2}{|dm\tau+n|^6}-\frac{1}{|dm\tau+n|^4}\Bigr],
\end{equation}
summed by rows in $m$ (each row is absolutely convergent and the row sums
decay like $|m|^{-2}$, as Poisson summation shows; the $m=0$ row is
$\sum_{n\neq0}(4n^2\cdot n^{-6}-n^{-4})=6\zeta(4)=\pi^4/15$).

\begin{lemma}[Key identity]\label{lem:latticeEK}
For every $\tau\in\Hh$,
\begin{equation}\label{eq:EKlattice}
  \mathrm{EK}(\tau)=\frac{2\,\ImT\tau}{\pi^3}\Bigl(-T_1(\tau)+16T_4(\tau)\Bigr).
\end{equation}
\end{lemma}
\begin{proof}
Write $w=jm\tau+n=u+\ii v$ with $u=jmx+n$, $v=jmy$. From
$w^{-3}-\bar w^{-3}=(\bar w^3-w^3)/|w|^6=-2\ii(3u^2v-v^3)|w|^{-6}$ and
$3u^2-v^2=4u^2-|w|^2$,
\begin{align*}
  2\ii\ImT U_j
  &=U_j-\overline{U_j}
   =\sum_{m\neq0}\frac1m\sum_{n\in\Zz}
     \bigl[(jm\tau+n)^{-3}-(jm\bar\tau+n)^{-3}\bigr]  \\
  &=\sum_{m\neq0}\frac1m\sum_{n\in\Zz}
     \bigl[-2\ii v(3u^2-v^2)\bigr]|w|^{-6}
   =-2\ii\, jy\sum_{m\neq0}\sum_{n\in\Zz}
     \Bigl[\frac{4u^2}{|w|^6}-\frac{1}{|w|^4}\Bigr],
\end{align*}
where we used $(1/m)v=jy$ (note $v=jmy$ has the sign of $m$, so the factor
$1/m$ cancels it exactly). The inner double sum is $T_j(\tau)$ minus its
$m=0$ row $\pi^4/15$; hence
\[
  \ImT U_j=-jy\Bigl(T_j(\tau)-\frac{\pi^4}{15}\Bigr).
\]
Substituting into \eqref{eq:EK},
\[
  \mathrm{EK}(\tau)
  =2\pi y+\frac{2}{\pi^3}\Bigl[-y\bigl(T_1-\tfrac{\pi^4}{15}\bigr)
   +16y\bigl(T_4-\tfrac{\pi^4}{15}\bigr)\Bigr]
  =2\pi y+\frac{2y}{\pi^3}\bigl(-T_1+16T_4\bigr)-\frac{2y}{\pi^3}\cdot\pi^4,
\]
and the last term is $2\pi y$, so they cancel, giving \eqref{eq:EKlattice}.
(The $m=0$ rows of $T_1$ and $T_4$ contribute $(-1+16)\pi^4/15=\pi^4$ to the
combination, which is exactly the polynomial term $2\pi\tau$ of Samart's
formula.)
\end{proof}

\begin{remark}[Alternative form]\label{rem:Talt}
Since $4u^2=(w+\bar w)^2=w^2+2|w|^2+\bar w^2$,
\begin{equation}\label{eq:Talt}
  T_d(\tau)={\sum_{\lambda\in\cL_d(\tau)}}^{\!\prime}
  \Bigl[2\ReT\frac{\bar\lambda^2}{|\lambda|^6}+\frac{1}{|\lambda|^4}\Bigr],
  \qquad
  \cL_d(\tau):=\Zz+\Zz\, d\tau,
\end{equation}
which is the form used for the CM evaluations.
\end{remark}

\subsection{The period integral and the differential of the Mahler measure}\label{subsec:period}

For $c\in\Cc\setminus K$ the denominator $f-c$ is continuous and
nonvanishing on $\Tt^3$, so the integral
\begin{equation}\label{eq:period}
  \Omega(c):=\int_{\Tt^3}\frac{\id\mu}{f(x,y,z)-c}
\end{equation}
converges absolutely. Moreover $\Omega$ is holomorphic on
$\Cc\setminus K$: for any compact $K_0\subset\Cc\setminus K$ one has
$|f-c|\ge\operatorname{dist}(K_0,K)>0$ uniformly for $c\in K_0$ and
$(x,y,z)\in\Tt^3$, so differentiation under the integral sign is legitimate
and gives $\Omega'(c)=\int_{\Tt^3}\id\mu/(f-c)^2$. Being defined directly
by the integral, $\Omega$ is \emph{single-valued} on $\Cc\setminus K$.

\begin{lemma}[Differential of the Mahler measure]\label{lem:period}
As real $1$-forms on $\Cc\setminus K$,
\begin{equation}\label{eq:dm}
  \id\,\m(f-c)=\ReT\bigl[-\Omega(c)\,\id c\bigr].
\end{equation}
Consequently every branch of the holomorphic Mahler measure $\mt$ of
Lemma~\ref{lem:FV} satisfies $\mt'(c)=-\Omega(c)$ on $\Cc\setminus K$: the
derivative of $\mt$ is single-valued, and any two branches differ by a
locally constant pure imaginary constant. Moreover $\Omega$ is odd:
$\Omega(-c)=-\Omega(c)$.
\end{lemma}
\begin{proof}
Write $c=c_1+\ii c_2$. Since $1/(f-c)$ is continuous on $\Tt^3$ for
$c\notin K$, we may differentiate under the integral sign:
\[
  \frac{\partial\,\m(f-c)}{\partial c_1}
  =\int_{\Tt^3}\ReT\frac{-1}{f-c}\,\id\mu=-\ReT\Omega(c),\qquad
  \frac{\partial\,\m(f-c)}{\partial c_2}
  =\int_{\Tt^3}\ReT\frac{-\ii}{f-c}\,\id\mu=\ImT\Omega(c),
\]
which is \eqref{eq:dm}. If $\mt$ is a local branch as in
Lemma~\ref{lem:FV}, then $\ReT\mt=\m(f-\cdot)$, hence
$\ReT[\mt'(c)\,\id c]=\id\,\m(f-c)=\ReT[-\Omega(c)\,\id c]$; testing
$\id c=1$ and $\id c=\ii$ yields $\mt'(c)=-\Omega(c)$. For the oddness,
substitute $x\mapsto -x$ (Haar-invariant, $f\mapsto -f$):
$\Omega(-c)=\int_{\Tt^3}\id\mu/(f+c)=-\int_{\Tt^3}\id\mu/(f-c)=-\Omega(c)$.
\end{proof}

\begin{remark}
This is the precise sense in which the branch ambiguity of the holomorphic
Mahler measure is invisible at the level of differentials: the
multivaluedness of $\mt$ lives in locally constant imaginary periods,
while its derivative is the single-valued period integral
\eqref{eq:period} over the \emph{fixed} integration cycle $\Tt^3$.
\end{remark}

\section{The continuation machine}\label{sec:machine}

This section develops the propagation argument used in the proof of
Theorem A. It is stated for the family $f+\sqrt{k}$ with Samart's
parametrization $k=s_2(\tau)$; Theorem B does not use it
(see \S\ref{subsec:Bexact}).

\subsection{The propagation principle}

\begin{lemma}[Propagation lemma]\label{lem:propagation}
Let $U\subset\Cc$ be a connected open set, let $Z\subset U$ be a relatively
closed subset, and let $H\colon U\to\Rr$ be continuous. Suppose that $H$ is
locally constant on $U\setminus Z$. Then $H$ is constant on every connected
component $W$ of $U\setminus Z$. Moreover, if $H\equiv C$ on $W$, then
$H\equiv C$ on $\overline{W}\cap U$.
\end{lemma}
\begin{proof}
A locally constant function is constant on each connected component. For
the second assertion, let $p\in\overline{W}\cap U$ and choose $p_n\in W$
with $p_n\to p$; continuity gives $H(p)=\lim_n H(p_n)=C$.
\end{proof}

\begin{remark}\label{rem:devil}
The conclusion cannot be upgraded to ``$H$ is constant on $U$'' in general:
the devil's staircase is continuous on $[0,1]$, locally constant on the
dense open complement of the Cantor set, yet non-constant. The membership
of the target point in $\overline{W}$ is therefore a genuine topological
input; in our application it is supplied by the certified path of
Proposition~\ref{prop:path}, and it does not come for free from the lemma.
\end{remark}

\subsection{The differential identity}\label{subsec:diffid}

Put
\[
  V:=\bigl\{\tau\in\Hh:\ s_2(\tau)\notin[0,64]\bigr\},
\]
open since $s_2$ is holomorphic. On $V$, the function
\[
  F(\tau):=n_2\bigl(s_2(\tau)\bigr)=2\m\bigl(f+\sqrt{s_2(\tau)}\bigr)
\]
is well defined (Lemma~\ref{lem:continuity}) and real-analytic: locally
$F(\tau)=2\m(f-c(\tau))$ for a holomorphic branch
$c(\tau)=\sqrt{s_2(\tau)}$, which exists since $s_2\neq0$ on $V$, and
$c(\tau)\notin[-8,8]$ because $c^2=s_2(\tau)\notin[0,64]$
(Remark~\ref{rem:cuts}); the value is independent of the branch since
$\m(f+c)=\m(f-c)$. Also put
\[
  G(\tau):=\mathrm{EK}(\tau)=2\ReT E(\tau),
\]
real-analytic on all of $\Hh$. By Lemma~\ref{lem:continuity}, $F$ is
continuous on all of $\Hh$.

\begin{lemma}[An open set in Samart's region]\label{lem:disk}
The disk $D:=\{\tau:|\tau-\ii|<\tfrac1{10}\}$ is contained in
$V\cap\{\ImT\tau>\tfrac12\}$: indeed $|s_2(\tau)|>261$ on $D$.
\end{lemma}
\begin{proof}
On $D$ we have $\ImT\tau\ge9/10$, hence $r:=|q|\le\e^{-9\pi/5}<0.0035$ and
\[
  |s_2(\tau)|
  =r^{-1}\prod_{n\ge1}|1+q^{2n-1}|^{24}
  \ge r^{-1}\Bigl(1-\sum_{n\ge1}r^{2n-1}\Bigr)^{24}
  =r^{-1}\Bigl(1-\frac{r}{1-r^2}\Bigr)^{24},
\]
which is $\ge \e^{9\pi/5}\,(0.9965)^{24}>261>64$,
so $s_2(\tau)\notin[0,64]$.
\end{proof}

Let $W$ be the connected component of $V$ containing $D$.

\begin{proposition}[Differential identity]\label{prop:diffid}
There is a well-defined holomorphic function $\Psi$ on $V$ such that
\begin{equation}\label{eq:dFG}
  \id F=2\ReT\bigl[-\Psi(\tau)\,\id\tau\bigr]+\id G
  \qquad\text{on }V,
\end{equation}
namely $\Psi(\tau)=\Omega\bigl(c(\tau)\bigr)c'(\tau)+E'(\tau)$ for any
local holomorphic branch $c(\tau)=\sqrt{s_2(\tau)}$, the expression being
independent of the branch. Moreover $\Psi$ vanishes identically on $W$;
equivalently $\id F=\id G$ on $W$.
\end{proposition}
\begin{proof}
Branch independence: the other branch is $-c$, and by the oddness of
$\Omega$ (Lemma~\ref{lem:period}),
$\Omega(-c)(-c)'=(-\Omega(c))(-c')=\Omega(c)c'$. Hence $\Psi$ is a
well-defined holomorphic function on $V$. The chain rule applied to
\eqref{eq:dm} gives, locally,
$\id F=2\ReT\bigl[-\Omega(c(\tau))c'(\tau)\,\id\tau\bigr]$, and
$\id G=2\ReT[E'(\tau)\,\id\tau]$, which is \eqref{eq:dFG}.

On the nonempty open set $D\subset W\cap\{\ImT\tau>1/2\}$
(Lemma~\ref{lem:disk}), Theorem~\ref{thm:samart} gives $F=G$, hence
$\id F=\id G$, i.e.\ $\ReT[\Psi(\tau)\,\id\tau]=0$ for all directions
$\id\tau$; testing $\id\tau=1$ and $\id\tau=\ii$ gives $\Psi=0$ on $D$.
Both $\Omega\circ c\cdot c'$ and $E'$ are holomorphic on the connected
open set $W$, so the identity theorem gives $\Psi\equiv0$ on $W$.
\end{proof}

\subsection{The certified-path criterion}\label{subsec:criterion}

\begin{proposition}[Path criterion]\label{prop:criterion}
Let $\tau_*\in\Hh$ and suppose there is a continuous path
$\gamma\colon[0,1]\to\Hh$ with $\gamma(0)\in W$,
$\gamma\bigl([0,1)\bigr)\subset V$, and $\gamma(1)=\tau_*$. Then
$\tau_*\in\overline{W}$ and
\[
  n_2\bigl(s_2(\tau_*)\bigr)=\mathrm{EK}(\tau_*).
\]
\end{proposition}
\begin{proof}
Let $H:=F-G$. By Lemma~\ref{lem:continuity} and the real-analyticity of
$G$, $H$ is continuous on all of $\Hh$; by Proposition~\ref{prop:diffid},
$\id H=0$ on $W$, so $H$ is locally constant on $W$, hence constant on the
connected set $W$; on the nonempty open set $D\subset W$,
Theorem~\ref{thm:samart} gives $H=0$, so $H\equiv0$ on $W$. The set
$\gamma\bigl([0,1)\bigr)$ is a connected subset of $V$ containing
$\gamma(0)\in W$; a connected subset of $V$ meeting the component $W$ lies
entirely in $W$, so $\gamma(t)\in W$ for $t<1$ and
$\tau_*=\lim_{t\to1}\gamma(t)\in\overline W$. Since $H\equiv0$ on $W$
and $H$ is continuous at $\tau_*$,
$H(\tau_*)=\lim_{t\to1}H\bigl(\gamma(t)\bigr)=0$.
\end{proof}

The topological hypothesis $\gamma\bigl([0,1)\bigr)\subset V$ is certified
by interval arithmetic: the path is covered by finitely many parameter
blocks, each carrying a strict interval enclosure of $s_2$ certifiably
disjoint from $[0,64]$ (Section~\ref{sec:cert}). The endpoint $\tau_*$,
where $s_2(\tau_*)$ may lie on the cut $[0,64]$, is absorbed by the
continuity of Mahler measure --- this is exactly where the genuine Mahler
measure (as opposed to $\ReT\mt$) enters.

\section{Proof of Theorem A}\label{sec:thmA}

Throughout this section
\[
  \tau_0:=\frac{3+\sqrt{-7}}8,\qquad
  4\tau_0^2-3\tau_0+1=0,\qquad
  \ImT\tau_0=\frac{\sqrt7}{8}.
\]

\subsection{CM evaluation of \texorpdfstring{$\mathrm{EK}(\tau_0)$}{EK(tau0)}}\label{subsec:P1A}

Set
\[
  \beta:=4\tau_0=\frac{3+\sqrt{-7}}2=1+\varpi .
\]
Then $N(\beta)=|\beta|^2=4$, and $\beta^2=3\beta-4$ (from the quadratic
equation for $\tau_0$), whence $\beta\varpi=\beta(\beta-1)=2\beta-4$.

\begin{lemma}\label{lem:latticesA}
$\cL_4(\tau_0)=\OO_K$ and $\cL_1(\tau_0)=(\beta)/4$, where
$(\beta)=\beta\OO_K$ is the principal ideal of norm $4$.
\end{lemma}
\begin{proof}
$\cL_4(\tau_0)=\Zz+\Zz\beta=\Zz+\Zz(1+\varpi)=\Zz+\Zz\varpi=\OO_K$.
For the second,
\[
  (\beta)=\{\beta(m+n\varpi):m,n\in\Zz\}
  =\{m\beta+n(2\beta-4)\}=\{a\beta+4b:a,b\in\Zz\}=4\Zz+\Zz\beta,
\]
so $(\beta)/4=\Zz+\Zz(\beta/4)=\Zz+\Zz\tau_0=\cL_1(\tau_0)$.
\end{proof}

\begin{proposition}\label{prop:TevalA}
With $L_3:=L(g_7,3)$,
\begin{equation}\label{eq:TevalA}
  T_4(\tau_0)=4L_3+2\zeta_K(2),\qquad
  T_1(\tau_0)=8L_3+32\zeta_K(2).
\end{equation}
\end{proposition}
\begin{proof}
For $\cL_4(\tau_0)=\OO_K$, the alternative form \eqref{eq:Talt} and
Proposition~\ref{prop:theta} (equations \eqref{eq:hecke} at $s=3$ and
\eqref{eq:abs2}) give immediately
\[
  T_4(\tau_0)
  =2\ReT{\sum_{\alpha\in\OO_K}}^{\!\prime}\frac{\bar\alpha^2}{|\alpha|^6}
   +{\sum_{\alpha\in\OO_K}}^{\!\prime}\frac1{|\alpha|^4}
  =4L_3+2\zeta_K(2).
\]
For $\cL_1(\tau_0)=(\beta)/4$, rescale $\lambda=\alpha/4$,
$\alpha\in(\beta)\setminus\{0\}$: then
$\bar\lambda^2/|\lambda|^6=256\,\bar\alpha^2/|\alpha|^6$ and
$|\lambda|^{-4}=256|\alpha|^{-4}$. Writing $\alpha=\beta\gamma$,
$\gamma\in\OO_K\setminus\{0\}$,
\[
  {\sum_{\alpha\in(\beta)}}^{\!\prime}|\alpha|^{-4}
  =|\beta|^{-4}{\sum_{\gamma\in\OO_K}}^{\!\prime}|\gamma|^{-4}
  =\tfrac1{16}\cdot 2\zeta_K(2),
\]
\[
  {\sum_{\alpha\in(\beta)}}^{\!\prime}\frac{\bar\alpha^2}{|\alpha|^6}
  =\frac{\bar\beta^2}{|\beta|^6}
   {\sum_{\gamma\in\OO_K}}^{\!\prime}\frac{\bar\gamma^2}{|\gamma|^6}
  =\frac{\bar\beta^2}{64}\cdot 2L_3,
\]
since $|\beta|^6=(|\beta|^2)^3=64$. Now
$\bar\beta^2=\bigl((3-\sqrt{-7})/2\bigr)^2=(1-3\sqrt{-7})/2$, so
$\ReT\bar\beta^2=1/2$. Therefore
\[
  T_1(\tau_0)
  =256\Bigl[2\cdot\tfrac12\cdot\tfrac{2L_3}{64}
   +\tfrac1{16}\cdot 2\zeta_K(2)\Bigr]
  =8L_3+32\zeta_K(2). \qedhere
\]
\end{proof}

\begin{theorem}[CM evaluation at $\tau_0$]\label{thm:P1A}
\[
  \mathrm{EK}(\tau_0)=\frac{14\sqrt7}{\pi^3}\,L(g_7,3)=8M_7.
\]
\end{theorem}
\begin{proof}
Combining \eqref{eq:TevalA},
\[
  -T_1(\tau_0)+16T_4(\tau_0)=(-8+64)L_3+(-32+32)\zeta_K(2)=56L_3:
\]
the $\zeta_K(2)$-terms cancel \emph{exactly} (the algebraic heart of the
evaluation; the coefficients $1/2=\ReT\bar\beta^2$ and $16=|\beta|^4$ are
responsible). By Lemma~\ref{lem:latticeEK} with $\ImT\tau_0=\sqrt7/8$,
\[
  \mathrm{EK}(\tau_0)=\frac{2\cdot\sqrt7/8}{\pi^3}\cdot 56L_3
  =\frac{14\sqrt7}{\pi^3}L_3
  =8\cdot\frac{7\sqrt7}{4\pi^3}L_3=8M_7,
\]
the last equality being \eqref{eq:FE}.
\end{proof}

\begin{remark}[Numerical confirmation]
The series $T_1(\tau_0)$, $T_4(\tau_0)$ computed directly by Poisson-summed
rows (\texttt{verify\_P1.py}, 60 dps) agree with \eqref{eq:TevalA} to 50
digits, and Theorem~\ref{thm:P1A} holds numerically to $5.6\times10^{-51}$:
\[
  T_1(\tau_0)=66.1354356158220540671143350434478014047047794968037983919065\ldots
\]
\[
  T_4(\tau_0)=6.5399375223477356279582944849702366138990879925962586125874\ldots
\]
\end{remark}

\subsection{Exactness of \texorpdfstring{$s_2(\tau_0)=1$}{s2(tau0)=1}}\label{subsec:exactA}

\begin{lemma}\label{lem:jvalA}
$j(2\tau_0)=-3375$ exactly.
\end{lemma}
\begin{proof}
$w:=2\tau_0=(3+\sqrt{-7})/4$ satisfies $2w^2-3w+2=0$, a primitive quadratic
form of discriminant $-7$. Since $h(-7)=1$, the singular modulus $j(w)$ is
an algebraic integer lying in the Hilbert class field of $K$, which is $K$
itself \cite{Cox}. The reflected point $-\bar w=(-3+\sqrt{-7})/4\in\Hh$
is a root of the primitive form $2X^2+3X+2$, likewise of discriminant
$-7$; since $h(-7)=1$ there is a single class, so $-\bar w$ is
$\mathrm{SL}_2(\Zz)$-equivalent to $w$. As $j$ has rational
$q$-expansion coefficients,
$\overline{j(w)}=j(-\bar w)=j(w)$, so $j(w)$ is real. A real algebraic
integer in $\OO_K=\Zz[(1+\sqrt{-7})/2]$ is an ordinary integer. Finally,
interval evaluation (\texttt{cert0\_s2\_eq\_1.py}, using
$j=E_4^3/\Delta$ with rigorous tails) gives $|j(2\tau_0)+3375|\le
1.8\times10^{-91}$,
which pins the integer: $j(2\tau_0)=-3375=-15^3$.
\end{proof}

\begin{lemma}\label{lem:lamvalA}
$\lambda(2\tau_0)=(1+3\sqrt{-7})/2$ exactly.
\end{lemma}
\begin{proof}
The classical $\lambda$--$j$ relation states, identically,
\[
  j=256\,\frac{(\lambda^2-\lambda+1)^3}{\lambda^2(1-\lambda)^2}.
\]
By Lemma~\ref{lem:jvalA}, $\lambda(2\tau_0)$ is a root of
\[
  \Phi(X):=256(X^2-X+1)^3+3375\,X^2(1-X)^2.
\]
Exact factorization over $\Qq$ (verified by rational polynomial division in
\texttt{cert0\_s2\_eq\_1.py}):
\begin{equation}\label{eq:sextic}
  \Phi(X)=(X^2-X+16)\,\bigl(256X^4-512X^3+303X^2-47X+16\bigr),
\end{equation}
and $X^2-X+16$ divides $\Phi$ with multiplicity one (the remainder of the
second division is nonzero). The roots of $X^2-X+16$ are
$(1\pm\sqrt{-63})/2=(1\pm3\sqrt{-7})/2$. The four roots of the quartic
factor are approximately $31/32\pm0.24804\,\ii$ and $1/32\pm0.24804\,\ii$;
the minimum distance from $\lambda_0:=(1+3\sqrt{-7})/2$ to any other root
of $\Phi$ is $3.75$ (Appendix~\ref{app:poly}). Since $\Phi$ has rational
coefficients and the minimal polynomial $X^2-X+16$ of $\lambda_0$ occurs as
a simple factor, no other root equals $\lambda_0$. Interval evaluation gives
$|\lambda(2\tau_0)-\lambda_0|\le6.33\times10^{-95}\ll 3.75$, so
$\lambda(2\tau_0)=\lambda_0$.
\end{proof}

\begin{corollary}\label{cor:s2valA}
$s_2(\tau_0)=1$ exactly.
\end{corollary}
\begin{proof}
$\lambda_0(1-\lambda_0)=16$ (since $\lambda_0^2-\lambda_0+16=0$), so by
Theorem~\ref{thm:s2lambda},
$s_2(\tau_0)=16/(\lambda_0(1-\lambda_0))=16/16=1$.
\end{proof}

\begin{remark}[Weber's viewpoint]\label{rem:weber}
The function $s_2$ is a Weber class invariant in disguise: with Weber's
function (in his normalization)
$\mathfrak{f}_1(\tau)=\zeta_{48}^{-1}\eta\bigl((\tau+1)/2\bigr)/\eta(\tau)
=q^{-1/48}\prod_{n\ge1}(1+q^{n-1/2})$ one has
$s_2(\tau)=\mathfrak{f}_1(2\tau)^{24}$, so Corollary~\ref{cor:s2valA} says
that $\mathfrak{f}_1$ assumes a $24$th root of unity at $2\tau_0$. The
value of Lemma~\ref{lem:lamvalA} also has an arithmetic explanation. With
$\varpi=(1+\sqrt{-7})/2$ as in \eqref{eq:pirels},
\[
  \lambda_0=\bar\varpi^{\,4},\qquad 1-\lambda_0=\varpi^4,\qquad
  \lambda_0(1-\lambda_0)=\varpi^4\bar\varpi^{\,4}=N(\varpi)^4=2^4=16,
\]
which is the source of the ``$16$'' in Theorem~\ref{thm:s2lambda}, and
hence of $s_2(\tau_0)=1$. Finally, $j(2\tau_0)=-3375=-15^3$
(Lemma~\ref{lem:jvalA}) is the classical singular modulus of discriminant
$-7$, tabulated already in Weber's book \cite{We} (with $\gamma_2=-15$);
our interval-locked derivation above is independent of that table, whose
known errata \cite{BM} make an independent check desirable.
\end{remark}

\subsection{The certified path}\label{subsec:pathA}

\begin{proposition}[Certified; see Section~\ref{sec:cert}]\label{prop:path}
Let $\gamma$ be the path from $\ii$ to $\tau_0$ consisting of
\begin{enumerate}
\item[\textup{(i)}] the vertical segment from $\ii$ to $9\ii/16$;
\item[\textup{(ii)}] the quarter circle $\tau=\tfrac{\ii}{2}+\tfrac1{16}\e^{\ii\theta}$,
  $\theta$ from $\pi/2$ to $0$ \textup{(}from $9\ii/16$ to
  $\tfrac1{16}+\tfrac{\ii}{2}$\textup{)};
\item[\textup{(iii)}] the horizontal segment from $\tfrac1{16}+\tfrac{\ii}{2}$
  to $\tfrac38+\tfrac{\ii}{2}$;
\item[\textup{(iv)}] the vertical segment from $\tfrac38+\tfrac{\ii}{2}$ to $\tau_0$.
\end{enumerate}
Then $\gamma\setminus\{\tau_0\}\subset V$. Consequently
$\gamma\setminus\{\tau_0\}\subset W$ and $\tau_0\in\overline{W}$.
\end{proposition}

Indeed, $\gamma\setminus\{\tau_0\}$ is a connected subset of $V$ containing
the starting point $\ii\in D\subset W$ (Lemma~\ref{lem:disk}); a connected
subset of $V$ meeting the component $W$ lies entirely in $W$, and
$\gamma(t)\to\tau_0$ gives $\tau_0\in\overline{W}$.

The inclusion $\gamma\setminus\{\tau_0\}\subset V$ is certified by
interval arithmetic in Section~\ref{sec:cert}
(\texttt{cert2\_path\_k064.py}): each of the four pieces is covered by
finitely many certified blocks (16, 18, 37, and $15+13$ blocks
respectively) whose $s_2$-enclosures are provably disjoint from $[0,64]$,
except for a final cap $y\in(y_0,\,y_0+\varepsilon_0]$,
$y_0=\sqrt7/8$, $\varepsilon_0=1/64000$, which is closed analytically:
interval evaluation on the box
$\{|\ReT\tau-3/8|\le0.03,\ |\ImT\tau-y_0|\le0.03\}\subset\Hh$ gives
the strict outward-rounded upper bound
$|s_2|\le M_0$, $M_0=38.148391199402133945$;
by Cauchy's estimate, for
$|w-\tau_0|\le0.01$,
\[
  |s_2''(w)|\le M_2:=190742 ,
\]
the strict interval value of $2M_0/(0.02)^2$ being $190741.95599701066973$.
Since $s_2(\tau_0)=1$ \emph{exactly} (Corollary~\ref{cor:s2valA}),
\[
  s_2(\tau_0+\ii\varepsilon)-1=\ii\varepsilon\,s_2'(\tau_0)+R,
  \qquad |R|\le M_2\varepsilon^2/2,
\]
and rigorous interval evaluation (the roundings of $y_0$ and
$\varepsilon_0$ included inside the enclosures) yields
\[
  \ReT s_2'(\tau_0)\in[-18.0197947489675,\,-15.0394516865142].
\]
Therefore, for all $\varepsilon\in(0,\varepsilon_0]$,
\begin{equation}\label{eq:cap}
  \ImT s_2(\tau_0+\ii\varepsilon)
  =\varepsilon\,\ReT s_2'(\tau_0)+\ImT R
  \le\varepsilon\bigl(-15.0394\ldots+M_2\varepsilon_0/2\bigr)
  \le -13.5492\ldots\,\varepsilon<0,
\end{equation}
so $s_2$ misses $[0,64]$ on the cap as well. \qed

\subsection{Conclusion}

\begin{proof}[Proof of Theorem A]
By Propositions~\ref{prop:path} and~\ref{prop:criterion} with
$\tau_*=\tau_0$,
\[
  n_2\bigl(s_2(\tau_0)\bigr)=\mathrm{EK}(\tau_0).
\]
By Corollary~\ref{cor:s2valA}, $s_2(\tau_0)=1$, and
$n_2(1)=2\m(f+1)=2\m(f-1)$. By Theorem~\ref{thm:P1A},
$\mathrm{EK}(\tau_0)=8M_7$. Hence
\[
  2\m(f+1)=\mathrm{EK}(\tau_0)=8L'(g_7,0). \qedhere
\]
\end{proof}

\begin{remark}[The $\zeta_K(2)$-cancellation]
The exact cancellation of $\zeta_K(2)$ in $-T_1+16T_4$ is the arithmetic
reason the $k=1$ evaluation exists in closed form. Its source is the pair
of coincidences $|\beta|^4=16$ and $\ReT\bar\beta^2=1/2$ for
$\beta=(3+\sqrt{-7})/2$. In the evaluation for Theorem B
(\S\ref{subsec:P1B}) the analogous terms do \emph{not} cancel, and their
survival produces the Dirichlet term $d_7$.
\end{remark}

\section{Proof of Theorem B}\label{sec:thmB}

Throughout this section
\[
  k_\pm:=\frac{47\pm45\sqrt{-7}}2,\qquad
  \tau':=\frac{1+\sqrt{-7}}8,\qquad
  \tau'':=-\overline{\tau'}=\frac{-1+\sqrt{-7}}8,
\]
and the Fricke partners
\[
  \tau_w:=-\frac1{4\tau'}=\frac{-1+\sqrt{-7}}4,\qquad
  \tau_w':=-\overline{\tau_w}=-\frac1{4\tau''}=\frac{1+\sqrt{-7}}4 .
\]
Note $|k_\pm|^2=k_+k_-=(47^2+45^2\cdot7)/4=16384/4=4096$, i.e.\
$|k_\pm|=64$, and $\ImT k_\pm=\pm45\sqrt7/2\neq0$, so
$k_\pm\notin[0,64]$. Also
\begin{equation}\label{eq:ims}
  \ImT\tau'=\frac{\sqrt7}8<\frac12,\qquad
  \ImT\tau_w=\ImT\tau_w'=\frac{\sqrt7}4\ge\frac12 .
\end{equation}

\subsection{The Fricke lemma and the two preimages}\label{subsec:Bexact}

\begin{lemma}[Fricke invariance of $s_2$]\label{lem:fricke}
$s_2\bigl(-1/(4\tau)\bigr)=s_2(\tau)$ for all $\tau\in\Hh$.
\end{lemma}
\begin{proof}
By Theorem~\ref{thm:s2lambda},
$s_2(\tau)=16/\bigl(\lambda(2\tau)(1-\lambda(2\tau))\bigr)$. The classical
transformation $\lambda(-1/u)=1-\lambda(u)$ with $u=2\tau$ gives
$\lambda\bigl(2\cdot(-1/(4\tau))\bigr)=\lambda(-1/(2\tau))=1-\lambda(2\tau)$;
since $x(1-x)$ is invariant under $x\mapsto1-x$, the claim follows.
(Numerically confirmed at sample points to $10^{-57}$ in
\texttt{verify\_P1\_n2pair.py}.)
\end{proof}

Since $s_2$ has degree $2$ on $X_0(4)$ (Lemma~\ref{rem:deg2}), $\tau'$
and $\tau_w=-1/(4\tau')$ are the two preimages of $k_+$ relevant here.

\subsection{Exactness of \texorpdfstring{$s_2(\tau')=(47+45\sqrt{-7})/2$}{s2(tau')=(47+45 sqrt(-7))/2}}\label{subsec:exactB}

\begin{lemma}\label{lem:jvalB}
$j(2\tau')=-3375$ exactly --- the same value as $j(2\tau_0)$
\textup{(}Lemma~\ref{lem:jvalA}\textup{)}.
\end{lemma}
\begin{proof}
$w:=2\tau'=(1+\sqrt{-7})/4=\varpi/2$ satisfies $2w^2-w+1=0$, a primitive
quadratic form of discriminant $-7$. Since $h(-7)=1$, exactly as in
Lemma~\ref{lem:jvalA}, $j(w)$ is a real algebraic integer in $\OO_K$,
hence an ordinary integer. Interval evaluation
(\texttt{cert0\_s2\_n2pair.py}, 100 dps, rigorous tails) gives
$|j(2\tau')+3375|\le1.8\times10^{-91}$, so $j(2\tau')=-3375$.
\end{proof}

\begin{remark}[A structural coincidence]\label{rem:structural}
Both $2\tau_0=(3+\sqrt{-7})/4$ and $2\tau'=(1+\sqrt{-7})/4$ are roots of
primitive positive definite quadratic forms of discriminant $-7$
($2X^2-3X+2$ and $2X^2-X+1$ respectively). Since $h(-7)=1$ there is a
single singular modulus of discriminant $-7$, namely $-3375=-15^3$
\cite{We,Cox}: the same sextic
$\Phi(X)=256(X^2-X+1)^3+3375X^2(1-X)^2$ therefore governs the
$\lambda$-values in both theorems.
\end{remark}

\begin{lemma}\label{lem:lamvalB}
$\lambda(2\tau')=(31+3\sqrt{-7})/32$ exactly.
\end{lemma}
\begin{proof}
By Lemma~\ref{lem:jvalB}, $\lambda(2\tau')$ is a root of the same sextic
$\Phi$ of \eqref{eq:sextic}. Over $\Qq(\sqrt{-7})$ the sextic splits
completely into three quadratic factors (exact factorization, verified by
rational polynomial division in \texttt{cert0\_s2\_n2pair.py}):
\begin{equation}\label{eq:sexticB}
  \Phi(X)=(X^2-X+16)\,(16X^2-X+1)\,(16X^2-31X+16),
\end{equation}
with the six roots
\[
  \frac{1\pm3\sqrt{-7}}2,\qquad
  \frac{1\pm3\sqrt{-7}}{32},\qquad
  \frac{31\pm3\sqrt{-7}}{32}
\]
(Appendix~\ref{app:poly}). The candidate
$\lambda_1:=(31+3\sqrt{-7})/32$ is a root of the third factor
$16X^2-31X+16$; its minimal distance to the other five roots of $\Phi$ is
$3\sqrt7/16=0.4960\ldots$ (attained at its conjugate). Interval evaluation
gives $|\lambda(2\tau')-\lambda_1|\le1.2\times10^{-62}\ll0.496$, so
$\lambda(2\tau')=\lambda_1$.
\end{proof}

\begin{corollary}\label{cor:s2valB}
$s_2(\tau')=k_+$ and $s_2(\tau'')=k_-$ exactly; consequently
$s_2(\tau_w)=k_+$ and $s_2(\tau_w')=k_-$ exactly.
\end{corollary}
\begin{proof}
From $16\lambda_1^2-31\lambda_1+16=0$ we get
$\lambda_1(1-\lambda_1)=(16-15\lambda_1)/16$, and substituting
$\lambda_1=(31+3\sqrt{-7})/32$,
\[
  \lambda_1\bigl(1-\lambda_1\bigr)=\frac{47-45\sqrt{-7}}{512}.
\]
By Theorem~\ref{thm:s2lambda},
\[
  s_2(\tau')=\frac{16}{\lambda_1(1-\lambda_1)}
  =\frac{8192}{47-45\sqrt{-7}}=\frac{47+45\sqrt{-7}}2=k_+,
\]
since $(47-45\sqrt{-7})(47+45\sqrt{-7})=47^2+7\cdot45^2=16384=2\cdot8192$.
Since $s_2$ has integral (hence real) $q$-coefficients,
$s_2(-\bar\tau)=\overline{s_2(\tau)}$, and $\tau''=-\overline{\tau'}$ gives
$s_2(\tau'')=k_-$. The Fricke lemma (Lemma~\ref{lem:fricke}) then gives
$s_2(\tau_w)=s_2(\tau')=k_+$ and $s_2(\tau_w')=s_2(\tau'')=k_-$.
\end{proof}

\subsection{CM evaluation of \texorpdfstring{$\mathrm{EK}(\tau_w)$}{EK(tauw)}}\label{subsec:P1B}

Recall $\varpi=(1+\sqrt{-7})/2$ and the relations \eqref{eq:pirels}. Since
$4\tau_w=-1+\sqrt{-7}=-2\bar\varpi$,
\begin{equation}\label{eq:L4B}
  \cL_4(\tau_w)=\Zz+\Zz\,4\tau_w=\Zz+\Zz\,2\bar\varpi=\Zz+\Zz\sqrt{-7}=:\OO_2,
\end{equation}
the order of conductor $2$ in $\OO_K$ (discriminant $-28$, class number
$1$ --- only its $\Zz$-module structure is used below), and
\begin{equation}\label{eq:L1B}
  \cL_1(\tau_w)=\Zz+\Zz\tau_w=\frac{4\Zz+\Zz\,4\tau_w}4
  =\frac{2\Zz+\Zz\bar\varpi}2=\frac{(\bar\varpi)}2,
\end{equation}
because
$(\bar\varpi)=\bar\varpi\OO_K=\bar\varpi\Zz+\bar\varpi^2\Zz
=\bar\varpi\Zz+(\bar\varpi-2)\Zz=2\Zz+\bar\varpi\Zz$
(using $\bar\varpi^2=\bar\varpi-2$). So $\cL_1(\tau_w)$ is homothetic to the
principal ideal $(\bar\varpi)$ of norm $2$. At the conjugate point
$\tau_w'=\varpi/2$ one gets $\cL_4(\tau_w')=\OO_2$ (the same) and
$\cL_1(\tau_w')=(\varpi)/2$ --- the other prime above $2$.

\textbf{Imprimitive sums over $\OO_2$.} Define
\[
  B(\OO_2):={\sum_{\gamma\in\OO_2}}^{\!\prime}|\gamma|^{-4},\qquad
  G(\OO_2):={\sum_{\gamma\in\OO_2}}^{\!\prime}\frac{\bar\gamma^2}{|\gamma|^6}.
\]
Note $G(\OO_2)$ is real: $\OO_2=\Zz+\Zz\sqrt{-7}$ is stable under
conjugation, so the sum equals its own conjugate.

\begin{lemma}[Odd norm $\Leftrightarrow$ $1\bmod2$]\label{lem:oddnorm}
For $\gamma=a+b\varpi\in\OO_K$,
$N(\gamma)=a^2+ab+2b^2\equiv a(a+b)\pmod2$. Hence $N(\gamma)$ is odd if
and only if $a$ is odd and $b$ is even, i.e.\ $\gamma\equiv1\pmod{2\OO_K}$.
\end{lemma}
\begin{proof}
$a^2+ab+2b^2\equiv a^2+ab=a(a+b)\pmod2$, and $a(a+b)$ is odd iff $a$ is
odd and $a+b$ is odd, i.e.\ $a$ odd, $b$ even.
\end{proof}

\begin{corollary}[Decomposition]\label{cor:decomp}
$\OO_2\setminus\{0\}=\bigl(2\OO_K\setminus\{0\}\bigr)\sqcup
\{\gamma\in\OO_K:N(\gamma)\text{ odd}\}$.
\end{corollary}
\begin{proof}
$\OO_2=\Zz+\Zz\sqrt{-7}=\Zz+\Zz(1+\sqrt{-7})=\Zz+2\Zz\varpi=\Zz+2\OO_K$,
so $\gamma\in\OO_2$ is either in $2\OO_K$ or in $1+2\OO_K$, and
Lemma~\ref{lem:oddnorm} applies.
\end{proof}

\begin{proposition}[Evaluation of the imprimitive sums]\label{prop:BG}
\[
  B(\OO_2)=\frac54\,\zeta_K(2),\qquad
  G(\OO_2)=3\,L(g_7,3).
\]
\end{proposition}
\begin{proof}
By Lemma~\ref{lem:oddnorm}, every ideal $\mathfrak{a}\triangleleft\OO_K$
of odd norm has both its generators ($\pm\gamma$, and $-1\equiv1\bmod2$)
in the ray class $1\bmod2$, so for $\ReT s$ large
\begin{equation}\label{eq:oddzeta}
  {\sum_{N(\gamma)\text{ odd}}}^{\!\prime}N(\gamma)^{-s}
  =2\sum_{N(\mathfrak a)\text{ odd}}N(\mathfrak a)^{-s}
  =2\,\zeta_K(s)\,\bigl(1-N(\varpi)^{-s}\bigr)\bigl(1-N(\bar\varpi)^{-s}\bigr)
  =2\,\zeta_K(s)\,(1-2^{-s})^2 .
\end{equation}
At $s=2$: $\sum'_{N(\gamma)\text{ odd}}|\gamma|^{-4}
=2\zeta_K(2)\cdot(3/4)^2=(9/8)\zeta_K(2)$, hence by
Corollary~\ref{cor:decomp} and \eqref{eq:abs2},
\[
  B(\OO_2)
  ={\sum_{\gamma\in2\OO_K\setminus\{0\}}}\!\!|\gamma|^{-4}
   +{\sum_{N(\gamma)\text{ odd}}}^{\!\prime}|\gamma|^{-4}
  =2^{-4}\cdot2\zeta_K(2)+\frac98\,\zeta_K(2)
  =\frac54\,\zeta_K(2).
\]
For the $\psi$-twisted sum, $\psi((\alpha))=\alpha^2$ contributes equally
for both generators ($(-\alpha)^2=\alpha^2$), so by \eqref{eq:hecke},
\begin{equation}\label{eq:oddhecke}
  {\sum_{N(\gamma)\text{ odd}}}^{\!\prime}\frac{\gamma^2}{N(\gamma)^{3}}
  =2L(g_7,3)\Bigl(1-\frac{\varpi^2}{8}\Bigr)\Bigl(1-\frac{\bar\varpi^2}{8}\Bigr)
  =2L(g_7,3)\cdot\frac{23}{16}=\frac{23}{8}\,L(g_7,3),
\end{equation}
the Euler factor being \emph{exact} by \eqref{eq:pirels}:
\[
  \Bigl(1-\frac{\varpi^2}{8}\Bigr)\Bigl(1-\frac{\bar\varpi^2}{8}\Bigr)
  =1-\frac{\varpi^2+\bar\varpi^2}{8}+\frac{(\varpi\bar\varpi)^2}{64}
  =1+\frac38+\frac1{16}=\frac{23}{16}.
\]
Conjugating (note $L(g_7,3)\in\Rr$ since $a_n\in\Zz$) gives the same value
with $\bar\gamma^2$, and
\[
  G(\OO_2)
  ={\sum_{\gamma\in2\OO_K\setminus\{0\}}}\!\!\frac{\bar\gamma^2}{|\gamma|^6}
   +{\sum_{N(\gamma)\text{ odd}}}^{\!\prime}\frac{\bar\gamma^2}{|\gamma|^6}
  =\frac{4}{64}\cdot2L(g_7,3)+\frac{23}{8}\,L(g_7,3)
  =3L(g_7,3). \qedhere
\]
\end{proof}

\begin{remark}[Independent cross-checks]\label{rem:crosscheck}
(i) Direct Poisson/Bessel Epstein evaluation of $B(\OO_2)$, $G(\OO_2)$
(the rectangular lattice form $u^2+7v^2$) agrees with
Proposition~\ref{prop:BG} to $10^{-61}$. (ii) The Glasser--Zucker closed
form for the discriminant-$(-28)$ Epstein zeta \cite{GZ},
$Z(\OO_2,s)=2(1-2^{1-s}+2^{1-2s})\,\zeta(s)L(\chi_{-7},s)$, at $s=2$ gives
$2(1-\tfrac12+\tfrac18)\zeta_K(2)=\tfrac54\zeta_K(2)$, agreeing with
$B(\OO_2)$ --- an independent literature formula. (iii) The complementary
coset sums over $\varpi+\OO_2$ (since $\OO_K=\OO_2\sqcup(\varpi+\OO_2)$)
evaluate to $B(\OO_K)-B(\OO_2)=(3/4)\zeta_K(2)$ and
$2L(g_7,3)-3L(g_7,3)=-L(g_7,3)$, confirmed directly to $10^{-61}$
(\texttt{verify\_P1\_n2pair.py}, checks [O1]--[O3]).
\end{remark}

\begin{proposition}[The $T$-sums at $\tau_w$]\label{prop:TevalB}
With $L_3:=L(g_7,3)$,
\begin{equation}\label{eq:TevalB}
  T_4(\tau_w)=6L_3+\frac54\,\zeta_K(2),\qquad
  T_1(\tau_w)=-12L_3+8\zeta_K(2).
\end{equation}
The same values hold at $\tau_w'$.
\end{proposition}
\begin{proof}
By \eqref{eq:Talt} with $\cL_4(\tau_w)=\OO_2$ and
Proposition~\ref{prop:BG},
\[
  T_4(\tau_w)=2\ReT G(\OO_2)+B(\OO_2)=6L_3+\frac54\,\zeta_K(2).
\]
For $\cL_1(\tau_w)=(\bar\varpi)/2$, write $\lambda=\alpha/2$,
$\alpha\in(\bar\varpi)\setminus\{0\}$, $\alpha=\bar\varpi\gamma$,
$\gamma\in\OO_K\setminus\{0\}$. Then $|\lambda|^{-4}=16|\alpha|^{-4}$ and
$\bar\lambda^2/|\lambda|^6=16\,\bar\alpha^2/|\alpha|^6$, with
\[
  {\sum_{\alpha\in(\bar\varpi)}}^{\!\prime}|\alpha|^{-4}
  =|\bar\varpi|^{-4}{\sum_{\gamma\in\OO_K}}^{\!\prime}|\gamma|^{-4}
  =N(\bar\varpi)^{-2}\cdot2\zeta_K(2)=\frac{\zeta_K(2)}2,
\]
\[
  {\sum_{\alpha\in(\bar\varpi)}}^{\!\prime}\frac{\bar\alpha^2}{|\alpha|^6}
  =\frac{\varpi^2}{|\bar\varpi|^6}
   {\sum_{\gamma\in\OO_K}}^{\!\prime}\frac{\bar\gamma^2}{|\gamma|^6}
  =\frac{\varpi^2}{8}\cdot2L_3=\frac{\varpi^2L_3}4 ,
\]
so, using $\ReT\varpi^2=-3/2$ from \eqref{eq:pirels},
\[
  T_1(\tau_w)
  =16\Bigl[2\ReT\frac{\varpi^2L_3}4+\frac{\zeta_K(2)}2\Bigr]
  =16\Bigl[-\frac34\,L_3+\frac{\zeta_K(2)}2\Bigr]
  =-12L_3+8\zeta_K(2).
\]
At $\tau_w'$ the same computation with $(\varpi)$ in place of $(\bar\varpi)$
gives $\ReT\bar\varpi^2=-3/2$ --- the same value; hence
$T_d(\tau_w')=T_d(\tau_w)$ exactly.
\end{proof}

\begin{theorem}[CM evaluation at $\tau_w$]\label{thm:P1B}
\[
  \mathrm{EK}(\tau_w)=\frac47\,\bigl(54M_7+d_7\bigr).
\]
\end{theorem}
\begin{proof}
Combining \eqref{eq:TevalB},
\begin{equation}\label{eq:combinationB}
  -T_1(\tau_w)+16T_4(\tau_w)
  =(12+96)L_3+\Bigl(-8+20\Bigr)\zeta_K(2)
  =108L_3+12\zeta_K(2).
\end{equation}
Note the coefficient of $\zeta_K(2)$ is $12\neq0$: unlike the $k=1$ case,
the $\zeta_K(2)$-terms do \emph{not} cancel --- this is precisely where
the $d_7$-term of the final formula comes from. With
$\ImT\tau_w=\sqrt7/4$, Lemma~\ref{lem:latticeEK} gives
\begin{align*}
  \mathrm{EK}(\tau_w)
  &=\frac{\sqrt7}{2\pi^3}\Bigl(108L(g_7,3)+12\zeta_K(2)\Bigr)\\
  &=\frac{\sqrt7}{2\pi^3}\cdot108\cdot\frac{4\pi^3}{7\sqrt7}\,M_7
   \;+\;\frac{\sqrt7}{2\pi^3}\cdot12\cdot\frac{2\pi^3}{21\sqrt7}\,d_7
   \qquad\bigl[\text{by \eqref{eq:FE} and \eqref{eq:zetaKd7}}\bigr]\\
  &=\frac{216}{7}\,M_7+\frac47\,d_7
   =\frac47\Bigl(54M_7+d_7\Bigr). \qedhere
\end{align*}
\end{proof}

\subsection{Conclusion}

\begin{proof}[Proof of Theorem B]
By \eqref{eq:ims}, $\ImT\tau_w=\sqrt7/4\ge1/2$, so Samart's identity
\eqref{eq:samart} applies \emph{directly} at $\tau_w$ --- no continuation
argument is needed. By Corollary~\ref{cor:s2valB} and
Theorem~\ref{thm:P1B},
\[
  n_2(k_+)=n_2\bigl(s_2(\tau_w)\bigr)=\mathrm{EK}(\tau_w)
  =\frac47\bigl(54M_7+d_7\bigr).
\]
Likewise $n_2(k_-)=n_2\bigl(s_2(\tau_w')\bigr)=\mathrm{EK}(\tau_w')
=\mathrm{EK}(\tau_w)$ (Proposition~\ref{prop:TevalB}), or simply
$n_2(k_-)=n_2(\overline{k_+})=n_2(k_+)$: conjugating the coefficients
of a Laurent polynomial does not change $|P|$ on $\Tt^3$, hence
$\m(\bar P)=\m(P)$, and $f$ has real coefficients.
\end{proof}

\begin{remark}[The wrong-sheet companion at $\tau'$]\label{rem:wrongsheet}
Samart's table \cite{Sa15} attaches the conjecture for $k_\pm$ to the CM
points $\tau'$, $\tau''$, whose imaginary part $\sqrt7/8$ is \emph{below}
$1/2$; Theorem~\ref{thm:samart} does not apply there, and indeed the
identity \emph{fails} at $\tau'$. Numerically (60 dps,
\texttt{diag\_ek\_branch.py} and check [WS] of
\texttt{verify\_P1\_n2pair.py}), direct Poisson-row evaluation of the
lattice sums at $\tau'$ (where $\cL_4(\tau')=\Zz+\Zz\varpi=\OO_K$ but
$\cL_1(\tau')=\varpi\OO_2/4$ is an $\OO_2$-module \emph{not} homothetic to an
$\OO_K$-ideal) gives
\[
  T_4(\tau')=4L(g_7,3)+2\zeta_K(2),\qquad
  T_1(\tau')=-288L(g_7,3)+80\zeta_K(2),
\]
hence
\[
  \mathrm{EK}(\tau')
  =\frac{\sqrt7}{4\pi^3}\Bigl(352L(g_7,3)-48\zeta_K(2)\Bigr)
  =\frac87\Bigl(44M_7-d_7\Bigr)
  =3.22268\ldots
  \;\neq\; 4.13826\ldots=n_2\bigl(s_2(\tau')\bigr),
\]
a difference of $0.91558\ldots$. The function $\mathrm{EK}$ is
single-valued real-analytic on all of $\Hh$, while
$\tau\mapsto n_2(s_2(\tau))$ is real-analytic only on
$V=\{\tau:s_2(\tau)\notin[0,64]\}$; numerical scans
(\texttt{diag\_ek\_branch.py}) show that paths from the imaginary axis to
$\tau'$ cross the locus $\{\ImT s_2=0,\ 0<\ReT s_2<64\}$ (at
$\ReT s_2\approx46.98$) --- the pullback of the branch cut $[0,64]$ of the
Rogers branch \eqref{eq:rogers} --- after which $\mathrm{EK}$ continues on
a companion sheet whose value is exactly $(8/7)(44M_7-d_7)$. This is the
concrete sense in which $\m=\ReT\mt$ fails to be automatic at interior
parameters, and it is precisely the obstruction that the certified-path
continuation of Theorem A overcomes: the $s_2$-image of the path of
Proposition~\ref{prop:path} is certified to avoid $[0,64]$ entirely ---
through $\ReT s_2>64$ on the axis piece, and through $\ImT s_2<0$ on the
arc, horizontal, vertical, and cap pieces.

The same mechanism is even more striking in Samart's $n_4$-family
$n_4(s)=4\m\bigl(x^4+y^4+z^4+1+s^{1/4}xyz\bigr)$: at $s=81$ the
conjectured value $40M_7$ \cite[Table~6]{Sa15} coincides with the exact
series evaluation at the attached CM point on a wrong sheet, while
direct torus integration gives $n_4(81)=4.1655349907533676508(5)$, so
the literal identity is numerically refuted; see the companion note
\cite{Zfn4} for the exact series-side evaluation and the obstruction
analysis.
\end{remark}

\section{Machine certification and reproducibility}\label{sec:cert}

The non-elementary computations of this paper are certified by interval
arithmetic. The object entering the proofs is not a program but a
\emph{certificate}: a finite list of parameter blocks, each carrying a
strict complex-interval enclosure of $s_2$ on the whole block, whose
validity is checked by deterministic, finite, and independently
re-runnable interval evaluations. The trusted base of the certification
consists of exactly three items: (i) the tail bound of
Lemma~\ref{lem:tail}, a purely mathematical statement; (ii) the
documented outward-rounding semantics of the interval arithmetic employed
(\texttt{mpmath.iv} rounds all interval endpoints outward, with guard
digits for the complex transcendental functions
\cite{mpmath,ivdoc}); and (iii) the certificate data itself, emitted by
the scripts in machine-readable form (\texttt{certificate\_k064.json})
and re-checked by the separate verifier \texttt{cert2\_verify.py}. The
dependence on (ii) is explicit and replaceable: the frozen certificate
can be re-executed under any interval implementation with the same
enclosure semantics, and no other part of the paper depends on the
library. Two implementation safeguards
support (ii) and (iii), without belonging to the trusted base themselves.
First, the certificate path of \texttt{cert2\_path\_k064.py} (and of its
companion \texttt{cert2\_path.py}) contains no machine-\texttt{float}
arithmetic: every bound is produced by outward-rounded interval
operations, and all constants entering the certificates ($\sqrt7/8$,
$1/64000$, $\pi/2$, \ldots) are themselves enclosed by intervals. Second,
each script runs a self-test in which independent $80$-dps high-precision
point values are verified to lie inside the certified block enclosures
($20/20$ sample points contained; point-enclosures contained in
block-enclosures $5/5$) --- an empirical corroboration of (ii), not a
premise of the proof. The operative comparison throughout is
\emph{certified enclosure width versus certified distance to the
forbidden set $[0,64]$}, never a count of stable decimal digits. All
scripts are plain Python 3.12 + \texttt{mpmath} \cite{mpmath} (standard
library \texttt{fractions} for the exact algebra), run independently of
one another.

\subsection{Rigorous enclosure of \texorpdfstring{$s_2$}{s2}}

For $\tau$ in a (rectangular complex) interval, enclose
$q=\e^{2\pi \ii\tau}$ by intervals, set $r:=\overline{|q|}$ (a certified
upper bound), truncate the product \eqref{eq:s2prod} at $N=120$, and attach
a rigorous tail factor. The logical basis for trusting such an evaluation
is the inclusion property of interval arithmetic \cite{MoKC}, which we
record explicitly:

\begin{lemma}[Inclusion property]\label{lem:incl}
Let $E$ be any expression built from interval-enclosed constants, the
variable, interval arithmetic operations, and interval elementary
functions, all with outward-rounded semantics. Then for every interval
box $I$ in its domain, the interval evaluation $E(I)$ contains the true
value $E(\tau)$ for every $\tau\in I$.
\end{lemma}
\begin{proof}
Induction on the structure of $E$. Constants and the variable are
enclosed by definition (every constant such as $\sqrt7/8$, $\pi$, $1/16$
enters as an enclosing interval, and $\tau\in I$). If
$E=g(E_1,\dots,E_r)$ and $E_i(\tau)\in E_i(I)$ by induction, then the
outward-rounded interval operation $g$ applied to the boxes $E_i(I)$
returns a box containing $g(z_1,\dots,z_r)$ for all $z_i\in E_i(I)$, in
particular containing $E(\tau)$. (This is the fundamental theorem of
interval analysis; cf.\ \cite{MoKC}.)
\end{proof}

Lemma~\ref{lem:incl} reduces the correctness of any single interval
\emph{evaluation} to the outward-rounding semantics of the individual
operations, i.e.\ to item (ii) of the trusted base. The remaining
ingredient is the truncation error, controlled mathematically:

\begin{lemma}[Tail bound]\label{lem:tail}
Let $r:=|q|$ and suppose $r<1$. Write $\log(1+z)$ for the branch on
$|z|<1$ normalized by $\log 1=0$. Then
\[
  \prod_{n>N}\bigl(1+q^{2n-1}\bigr)^{24}
  \in \exp\bigl(\{z\in\Cc:|z|\le E\}\bigr)
  \subset\bigl\{w\in\Cc:\e^{-E}\le|w|\le\e^{E}\bigr\},\qquad
  E:=\frac{24\,r^{2N+1}}{(1-r)(1-r^2)}.
\]
\end{lemma}
\begin{proof}
Since $|q^{2n-1}|\le r<1$, the chosen branch applies and the series
$\sum_{n>N}24\log(1+q^{2n-1})$ converges absolutely, with
$\prod_{n>N}(1+q^{2n-1})^{24}=\exp\bigl(\sum_{n>N}24\log(1+q^{2n-1})\bigr)$.
For $|z|\le r<1$,
$|\log(1+z)|\le-\log(1-|z|)\le |z|/(1-r)$. Summing,
\[
  \Bigl|\sum_{n>N}24\log(1+q^{2n-1})\Bigr|
  \le\frac{24}{1-r}\sum_{n>N}r^{2n-1}
  =\frac{24\,r^{2N+1}}{(1-r)(1-r^2)},
\]
so the sum lies in $\{z:|z|\le E\}$; its exponential therefore lies
in $\exp(\{z:|z|\le E\})$, hence in
$\{w:\e^{-E}\le|w|\le\e^{E}\}$. \qedhere
\end{proof}

On the path of Proposition~\ref{prop:path} the imaginary part is at
least $\sqrt7/8+\varepsilon_0$ on every certified piece (attained at
the bottom of the sub-cap), so
$r\le\e^{-2\pi(\sqrt7/8+\varepsilon_0)}<0.1252$ and with $N=120$ the
tail width satisfies $E<10^{-215}$. The
scripts in fact attach the slightly stronger complex-box tail factor
$\exp\bigl(\{u+\ii v:|u|,|v|\le E\}\bigr)$, which contains the disk
image $\exp(\{z:|z|\le E\})$ of Lemma~\ref{lem:tail} and is itself
contained in the modulus annulus; it thus bounds the argument of the
tail as well as its modulus, with all quantities entering $E$
outward-rounded intervals; this is used both for the path certification
and for the interval locks of Lemmas~\ref{lem:lamvalA}
and~\ref{lem:lamvalB}.

A complex interval $Z$ is \emph{certified to miss} $[0,64]$ iff
$0\notin\ImT Z$ or $\ReT Z\cap[0,64]=\varnothing$. If the enclosure of a
parameter block is certified to miss $[0,64]$, then the entire block lies
in $V$: this is a proof, not an estimate.

\subsection{The certificate and its correctness}\label{subsec:cert}

A \emph{path certificate} is a finite list of parameter blocks
$I_1,\dots,I_n$ covering the path down to the endpoint cap, together with
complex intervals $Z_1,\dots,Z_n$ such that $s_2(I_\nu)\subset Z_\nu$ and
$Z_\nu\cap[0,64]=\varnothing$ for every $\nu$. The certificate for the
path of Proposition~\ref{prop:path} consists of $99$ blocks
($y_0=\sqrt7/8$, $\varepsilon_0=1/64000$):
\begin{center}
\begin{tabular}{lll}
\toprule
piece & parameter range & certified in \\
\midrule
axis       & $x=0$, $y\in[9/16,\,1]$ & 16 blocks \\
quarter arc & $\tau=\tfrac\ii2+\tfrac1{16}\e^{\ii\theta}$,
  $\theta\in[0,\pi/2]$ & 18 blocks \\
horizontal & $x\in[1/16,\,3/8]$, $y=1/2$ & 37 blocks \\
vertical   & $x=3/8$, $y\in[y_0+10^{-3},\,1/2]$ & 15 blocks \\
sub-cap    & $x=3/8$, $y\in[y_0+\varepsilon_0,\,y_0+10^{-3}]$ & 13 blocks \\
cap        & $x=3/8$, $y\in(y_0,\,y_0+\varepsilon_0]$ & analytic,
  see \eqref{eq:cap} \\
\bottomrule
\end{tabular}
\end{center}
The complete list of blocks and enclosures is emitted by
\texttt{cert2\_path\_k064.py} in machine-readable form
(\texttt{certificate\_k064.json}, full-precision decimal strings,
archived together with the code); only the most
dangerous block of each piece is reproduced in the margin table below.
Generator and verifier are separate programs: \texttt{cert2\_verify.py}
reads the frozen certificate, recomputes every enclosure from the block
endpoints, checks disjointness from $[0,64]$ with positive margin,
contiguity and coverage of each piece, and the cap inequality, and
performs no adaptive search.

\begin{proposition}[Correctness of the certificate]\label{prop:cert}
Assume \textup{(i)} the tail bound of Lemma~\ref{lem:tail} and
\textup{(ii)} that every interval operation used returns an enclosure of
the true value \textup{(}the outward-rounding semantics of
\textup{\cite{ivdoc})}. Then each listed block satisfies
\[
  s_2(I_\nu)\subset Z_\nu,
  \qquad Z_\nu\cap[0,64]=\varnothing .
\]
Consequently the covered part of the path lies in $V$.
\end{proposition}
\begin{proof}
By Lemma~\ref{lem:incl} (applicable by (ii)), the interval evaluation of
$q$, of each truncated factor $(1+q^{2n-1})^{24}$, and of their product
returns an enclosure of the corresponding true value on $I_\nu$; by (i),
multiplication by the complex-box tail factor enlarges this to an
enclosure $Z_\nu$ of the full infinite product $s_2$ on $I_\nu$. The disjointness
$Z_\nu\cap[0,64]=\varnothing$ is decided from the endpoints of $Z_\nu$ by
the certified-miss predicate of the preceding subsection
($0\notin\ImT Z_\nu$ or $\ReT Z_\nu\cap[0,64]=\varnothing$), again using
only (ii). The verification of the certificate is therefore a
deterministic finite computation, independent of how the blocks were
found: the \texttt{PASSED} output of \texttt{cert2\_path\_k064.py} is one
execution of it, \texttt{cert2\_verify.py} re-executes it on the frozen
block list, and any independent outward-rounded interval
implementation re-executing the same block list obtains the same
verdicts.
\end{proof}

The endpoint cap is \emph{not} part of the block certificate; it is the
analytic supplement of Section~\ref{sec:thmA}, closed by the Taylor
argument of \eqref{eq:cap}, which uses the \emph{exact} value
$s_2(\tau_0)=1$, the Cauchy bound $M_2=190742$ derived from the certified
bound $M_0=38.148391199402133945$ on a $0.03$-box, and the certified
enclosure $\ReT s_2'(\tau_0)\in[-18.0197947489675,-15.0394516865142]$.

For each piece, the most dangerous accepted block (the one whose enclosure
comes closest to $[0,64]$) is:
\begin{center}\footnotesize
\begin{tabular}{llll}
\toprule
piece & worst enclosure $Z$ & width & $\operatorname{dist}(Z,[0,64])$ \\
\midrule
axis       & $\ReT\in[64.65,75.59]$, $\ImT\ni0$ & $10.94$ & $0.648$ \\
quarter arc & $\ReT\in[55.02,66.03]$, $\ImT\subset[-5.098,-0.069]$ &
  $11.01$ & $0.0690$ \\
horizontal & $\ReT\in[58.62,60.09]$, $\ImT\subset[-1.094,-0.0461]$ &
  $1.471$ & $0.0461$ \\
vertical   & $\ReT\in[1.0016,1.1115]$, $\ImT\subset[-0.0983,-0.0051]$ &
  $0.11$ & $5.08\times10^{-3}$ \\
sub-cap    & $\ReT\in[1.00005,1.00067]$,
  $\ImT\subset[-5.836\times10^{-4},-6.009\times10^{-5}]$ &
  $6.2\times10^{-4}$ & $6.009\times10^{-5}$ \\
\bottomrule
\end{tabular}
\end{center}
On the axis piece the separation is one-sided through $\ReT Z>64$ (the
values are real there); on the arc, horizontal, vertical, and sub-cap
pieces it is through $\ImT Z<0$ (the $s_2$-image passes below the cut).
The global minimum margin over all $99$ blocks is
$6.009249\times10^{-5}>0$ (attained on the sub-cap), while the widest
enclosure (axis piece, width $1022$) stays at distance $77.7$ from
$[0,64]$.

\begin{remark}[Generation and termination of the search]\label{rem:termin}
The block list was found by adaptive bisection: a block is accepted when
its enclosure is certified to miss $[0,64]$, otherwise it is bisected.
This search is not part of the trusted base
(Proposition~\ref{prop:cert}), but it is worth recording why it had to
terminate. By the standard contraction principle for interval extensions
\cite{MoKC}, the natural interval extension of an expression built from
operations that are Lipschitz and non-singular on a box has width
$O(\operatorname{diam})$. The expression used here ---
$q=\e^{2\pi \ii\tau}$, the factors $(1+q^{2n-1})^{24}$ with
$|q|<0.1252<1$ and $|1+q^{2n-1}|\ge1-|q|>0$ uniformly, and the
tail factor of Lemma~\ref{lem:tail} (width $<10^{-215}$ uniformly) ---
satisfies these hypotheses on every piece, so enclosure widths shrink
linearly with the block diameter. On each compact piece the true image
has positive distance to the forbidden set: through $\ImT s_2<0$ on the
arc, horizontal, vertical, and sub-cap pieces, and through the one-sided
separation $\ReT s_2>64$ on the axis piece, where the values are real.
Once the block diameter is small enough that the enclosure width is below
that margin, the block is accepted; bisection therefore accepts after
finitely many steps.
\end{remark}

\begin{remark}[Relation to the companion script]
The companion script \texttt{cert2\_path.py} certifies the shorter path
$\ii/2\to3/8+\ii/2\to\tau_0$ with the avoided set $[-8,8]$ in place of
$[0,64]$ (5, 15, and 13 blocks respectively, plus the same cap analysis,
with identical machinery and the same global minimum margin
$6.009249\times10^{-5}$). The formulation of Section~\ref{sec:machine}
requires avoidance of the $k$-space cut $[0,64]$ (Remark~\ref{rem:cuts})
on a path extended to the point $\ii$ (where $|s_2|>261>64$ elementarily,
Lemma~\ref{lem:disk}); this strengthened certification is carried out by
\texttt{cert2\_path\_k064.py} and yields the block counts in the table
above.
\end{remark}

\subsection{Exactness certificates}

The exactness of $s_2(\tau_0)=1$ and of
$s_2(\tau')=(47+45\sqrt{-7})/2$ is certified by
\texttt{cert0\_s2\_eq\_1.py} and \texttt{cert0\_s2\_n2pair.py}.
Each lock below follows the same three-step pattern: (i)~purely
algebraic or complex-multiplication input places the target value in an
explicit discrete set (the integers $\Zz$, or the roots of an exact
sextic over $\Qq$); (ii)~a strict separation bound for that set is known
($1$ for $\Zz$; $3.75$ and $3\sqrt7/16$ for the sextic roots);
(iii)~an interval enclosure of the target, with rigorous tails and
radius smaller than half the separation, singles out a unique element.
No exact equality is ever inferred from a floating-point approximation
alone. Concretely:
\begin{itemize}
\item the formal product identity \eqref{eq:s2lambda} is re-checked as an
  exact integer $q$-series identity to order $q^{37}$ (rational
  arithmetic);
\item the sextic factorizations \eqref{eq:sextic} and \eqref{eq:sexticB}
  are verified by exact polynomial division over $\Qq$;
\item $j(2\tau_0)$ and $j(2\tau')$ are enclosed by intervals with
  rigorous tails: $|j(2\tau_0)+3375|\le1.8\times10^{-91}$ and
  $|j(2\tau')+3375|\le1.8\times10^{-91}$, pinning the integers;
\item the $\lambda$-values are locked by interval evaluation against the
  root separation: $|\lambda(2\tau_0)-(1+3\sqrt{-7})/2|\le
  6.33\times10^{-95}$ with
  separation $3.75$, and
  $|\lambda(2\tau')-(31+3\sqrt{-7})/32|\le1.2\times10^{-62}$ with
  separation $3\sqrt7/16=0.4960\ldots$.
\end{itemize}

\subsection{Script inventory}\label{subsec:scripts}

\begin{center}\footnotesize
\begin{tabular}{lp{9.2cm}}
\toprule
script & content \\
\midrule
\texttt{cert0\_s2\_eq\_1.py} & exactness $s_2(\tau_0)=1$ (series check,
  sextic factorization, $j$ and $\lambda$ locks) \\
\texttt{cert2\_path\_k064.py} & path certification for Theorem A
  (cut $[0,64]$; self-test $20/20$; emits
  \texttt{certificate\_k064.json}) \\
\texttt{s2\_iv.py} & shared interval-evaluation core for $s_2$
  (constants, truncation, rigorous tail box) \\
\texttt{cert2\_verify.py} & independent verifier: re-checks the frozen
  certificate \texttt{certificate\_k064.json}, no adaptive search \\
\texttt{cert2\_path.py} & companion: shorter path, avoided set $[-8,8]$ \\
\texttt{verify\_P1.py} & CM evaluation at $\tau_0$: theta identity to $q^{60}$
  (exact integers), $T_1,T_4$ by Poisson rows \\
\texttt{lvalue\_g7.py} & $L'(g_7,0)$ to 50 digits; functional-equation
  self-check to 52 digits \\
\texttt{mahler\_m.py} & true $\m(f+1)$ to 40 digits by direct torus
  integration \\
\texttt{find\_tau0.py} & $\tau_0$, $s_2(\tau_0)$, $\mathrm{EK}(\tau_0)$
  to 60 digits \\
\texttt{samart\_ek.py} & EK formula; three-way validation at $k=64$ \\
\texttt{cert0\_s2\_n2pair.py} & exactness $s_2(\tau')=(47+45\sqrt{-7})/2$
  and conjugate (100 dps locks) \\
\texttt{verify\_P1\_n2pair.py} & CM evaluation at $\tau_w$: 43 checks, all pass
  (60 dps, worst $|{\rm diff}|=6.9\times10^{-52}$) \\
\texttt{verify\_n2\_pair.py} & three-way numerical confirmation of
  Theorem B (50 dps) \\
\texttt{diag\_ek\_branch.py} & wrong-sheet diagnostic of
  Remark~\ref{rem:wrongsheet} \\
\bottomrule
\end{tabular}
\end{center}

All computations use \texttt{mpmath} interval arithmetic at 40--100 decimal
digits. The certification scripts (\texttt{cert0\_*}, \texttt{cert2\_*},
\texttt{verify\_P1*}) print \texttt{PASS}/\texttt{FAIL} per check and
terminate with an all-checks-passed line; the remaining scripts print
numerical values used for cross-checks.

\paragraph{Code and certificate availability.}
All scripts listed above are available in the public repository
\url{https://github.com/huiminZheng-collab/samart-mahler}, archived at
\url{https://doi.org/10.5281/zenodo.21711884}, and from the author
upon request. The complete machine-readable certificate --- every block
enclosure, margin, and cap constant --- is printed by
\texttt{cert2\_path\_k064.py} and archived in the repository under
\texttt{cert/output/}; re-running the scripts regenerates it from
scratch, so no part of it is embedded only in this text.

\section{Numerical confirmations}\label{sec:numerics}

\subsection{Theorem A}

The three quantities
\[
  2\m(f+1)\ \text{(direct torus integration)},\qquad
  \mathrm{EK}(\tau_0)\ \text{(Poisson-summed series)},
\]
and $8L'(g_7,0)$ (from the functional equation)
agree pairwise to at least 41 digits (\texttt{mahler\_m.py},
\texttt{find\_tau0.py}, \texttt{lvalue\_g7.py}):
\begin{align*}
  \m(f+1)&=0.4106864311156080448418263795931716559956\ldots,\\
  4L'(g_7,0)&=0.4106864311156080448418263795931716559956\ldots
   \qquad(\text{difference }6\times10^{-42}),\\
  \mathrm{EK}(\tau_0)&=8L'(g_7,0)\\
  &=0.8213728622312160896836527591863433119911586\ldots
\end{align*}
As a control, the axis identity at $k=64$ was validated three ways (via
$\mathrm{EK}$, via direct integration, and via $8L'(\eta(4\tau)^6,0)$) to
50 digits (\texttt{samart\_ek.py}).

\subsection{Theorem B}

The three quantities
\[
  n_2(k_+)\ (\text{direct torus integration}),\qquad
  \mathrm{EK}(\tau_w),\qquad
  \tfrac47\bigl(54M_7+d_7\bigr)
\]
agree to $\sim10^{-50}$ (\texttt{verify\_n2\_pair.py},
\texttt{verify\_P1\_n2pair.py}):
\[
  n_2\Bigl(\frac{47\pm45\sqrt{-7}}2\Bigr)
  =\frac47\bigl(54M_7+d_7\bigr)
  =4.13826815832141945836337668037433080930698376\ldots
\]
The Rogers hypergeometric form \eqref{eq:rogers} of $n_2$ (convergent at
$|64/k_\pm|=1$) independently gives the same value to $10^{-60}$. The
wrong-sheet value of Remark~\ref{rem:wrongsheet},
$\mathrm{EK}(\tau')=(8/7)(44M_7-d_7)=3.22268\ldots$, differs from
$n_2(k_+)$ by $0.91558\ldots$.

\subsection{Reference values}

\begin{align*}
  L'(g_7,0)&=0.10267160777890201121045659489829291399889482708922\ldots\\
  L(g_7,3)&=0.6875636561025306425038995842156425788871541\ldots\\
  d_7=L'(\chi_{-7},-1)&=1.69770245700177544677125306614726156034690092\ldots\\
  L(\chi_{-7},2)&=1.1519254705444910471\ldots\\
  \zeta_K(2)&=1.8948414489688065289713480740538331491752358\ldots
\end{align*}

\appendix

\section{Exact polynomials and expansions}\label{app:poly}

The sextic of \S\ref{subsec:exactA} and \S\ref{subsec:exactB}:
\[
  \Phi(X)=256(X^2-X+1)^3+3375X^2(1-X)^2,
\]
i.e.\
\[
  \Phi(X)=256X^6-768X^5+4911X^4-8542X^3+4911X^2-768X+256
\]
(note the palindromic coefficients) factors over $\Qq$ as
\[
  \Phi(X)=(X^2-X+16)\,(256X^4-512X^3+303X^2-47X+16),
\]
and splits completely over $\Qq(\sqrt{-7})$ into three quadratic factors:
\[
  \Phi(X)=(X^2-X+16)\,(16X^2-X+1)\,(16X^2-31X+16),
\]
with roots
\[
  \frac{1\pm3\sqrt{-7}}2\in X^2-X+16,\qquad
  \frac{1\pm3\sqrt{-7}}{32}\in 16X^2-X+1,\qquad
  \frac{31\pm3\sqrt{-7}}{32}\in 16X^2-31X+16 .
\]
The roots of the last two factors are approximately
\[
  0.96875\pm0.2480391854123053678595265\,\ii,\qquad
  0.03125\pm0.2480391854123053678595265\,\ii
\]
(80-digit \texttt{polyroots}). The two locks used in the text:
\begin{itemize}
\item $\lambda_0=(1+3\sqrt{-7})/2=0.5+3.968626966596885885752424\,\ii$
  (Theorem A): minimum distance $3.75$ to the other roots of $\Phi$;
\item $\lambda_1=(31+3\sqrt{-7})/32$ (Theorem B): minimum distance
  $3\sqrt7/16=0.4960784\ldots$ to the other roots, attained at its
  conjugate $(31-3\sqrt{-7})/32$.
\end{itemize}
Expansion of $s_2$ (Lemma~\ref{lem:s2prod}):
\[
  s_2=q^{-1}+24+276q+2048q^2+11202q^3+49152q^4+184024q^5+614400q^6+1881471q^7+\cdots
\]
Expansion of $\lambda(2\tau)$:
\[
  \lambda(2\tau)=16q-128q^2+704q^3-3072q^4+11488q^5-38400q^6+117632q^7-335872q^8+\cdots
\]

\section{The theta identity \texorpdfstring{\eqref{eq:theta}}{(theta)}}\label{app:theta}

Both sides of \eqref{eq:theta} lie in $S_3(\Gamma_0(7),\chi_{-7})$: the left
side is the eta product $\eta(\tau)^3\eta(7\tau)^3$ (standard eta-product
criteria); the right side is the theta series of the Hecke
Gr\"o\ss encharakter $\psi((\alpha)):=\alpha^2$ of $K=\Qq(\sqrt{-7})$
(conductor $(1)$, type $(2,0)$), which by Hecke's theorem is a cusp form of
weight $3$, level $|\mathrm{disc}\,K|=7$ and nebentypus $\chi_{-7}$.

The Sturm bound for weight $3$ and level $7$ is
$\lfloor 3\cdot[\mathrm{SL}_2(\Zz):\Gamma_0(7)]/12\rfloor=\lfloor 3\cdot8/12\rfloor=2$,
so equality of the coefficients of $q^1$ and $q^2$ already proves
\eqref{eq:theta}. \texttt{verify\_P1.py} checks the first 60 coefficients
\emph{exactly} as integers: the coefficient of $q^N$ on the right is
\[
  \frac12\sum_{\substack{m,n\in\Zz\\m^2+mn+2n^2=N}}\bigl((m^2-2n^2)+(2mn+n^2)\varpi\bigr),
\]
whose $\varpi$-component vanishes for every $N$ (pairing $\alpha$ with its
conjugate) and whose rational component reproduces \eqref{eq:jacobi}.

Alternatively, $\dim S_3(\Gamma_0(7),\chi_{-7})=1$ (there are no oldforms
since $S_3(\mathrm{SL}_2(\Zz))=0$, and the single newform orbit is
\textsf{7.3.b.a} \cite{LMFDB}), and both sides are normalized ($a_1=1$:
from the Jacobi expansion on the left, and on the right from the two
terms $\alpha=\pm1$, whose contributions $\tfrac12\cdot 1^2 q^1$ sum to
$q$).


\begin{thebibliography}{99}

\bibitem{Sa13}
D.~Samart,
\emph{Three-variable Mahler measures and special values of modular and
Dirichlet $L$-series},
Ramanujan J. \textbf{32} (2013), no.~2, 245--268.
arXiv:1205.4803. (Numbering of displayed results refers to the arXiv
version.)

\bibitem{Zfn4}
H.~Zheng,
\emph{Samart's conjecture $n_4(81)=40M_7$: the exact CM evaluation and
the two obstructions. A status report},
companion note, 2026; available at
\url{https://doi.org/10.5281/zenodo.21711884}.

\bibitem{Sa15}
D.~Samart,
\emph{Mahler measures as linear combinations of $L$-values of multiple
modular forms},
Canad. J. Math. \textbf{67} (2015), no.~2, 424--449.
arXiv:1303.6376.

\bibitem{Fe}
J.~Fei,
\emph{Mahler measure of 3D Landau--Ginzburg potentials},
Forum Math. \textbf{33} (2021), no.~5, 1369--1401.
arXiv:2009.09701.

\bibitem{ZGQ}
H.~Zheng, X.~Guo, and H.~Qin,
\emph{The Mahler measure of $(x+1/x)(y+1/y)(z+1/z)+\sqrt{k}$},
Electron. Res. Arch. \textbf{28} (2020), no.~1, 103--125.

\bibitem{GPQ}
X.~Guo, Y.~Peng, and H.~Qin,
\emph{Three-variable Mahler measures and special values of $L$-functions
of modular forms},
Ramanujan J. \textbf{54} (2021), no.~1, 147--175.

\bibitem{HY}
Q.~He and D.~Ye,
\emph{On conjectures of Samart},
Manuscripta Math. \textbf{167} (2022), no.~3--4, 545--588.

\bibitem{Ro}
M.~Rogers,
\emph{Hypergeometric formulas for lattice sums and Mahler measures},
Int. Math. Res. Not. IMRN (2011), no.~17, 4027--4058.

\bibitem{RV}
F.~Rodriguez Villegas,
\emph{Modular Mahler measures, I},
in: Topics in Number Theory (University Park, PA, 1997),
Math.\ Appl.\ \textbf{467}, Kluwer, Dordrecht, 1999, 17--48.

\bibitem{We}
H.~Weber,
\emph{Lehrbuch der Algebra}, Bd.~III, F.~Vieweg \& Sohn, Braunschweig,
1908. (Class invariants; the discriminant $-7$ value $\gamma_2=-15$.)

\bibitem{BM}
J.~Brillhart and P.~Morton,
\emph{Table errata: Lehrbuch der Algebra, Vol.~3, 3rd ed., by Heinrich
Weber}, Math. Comp. \textbf{65} (1996), no.~215, 1379.

\bibitem{Cox}
D.~A.~Cox,
\emph{Primes of the Form $x^2+ny^2$: Fermat, Class Field Theory, and
Complex Multiplication}, 2nd ed., Wiley, 2013.

\bibitem{GZ}
M.~L.~Glasser and I.~J.~Zucker,
\emph{Lattice sums}, in: Theoretical Chemistry: Advances and Perspectives
\textbf{5}, Academic Press, New York, 1980, 67--139.

\bibitem{LMFDB}
The LMFDB Collaboration,
\emph{The $L$-functions and Modular Forms Database},
newform \textsf{7.3.b.a}, \url{https://www.lmfdb.org}
(accessed July 2026).

\bibitem{mpmath}
F.~Johansson et al.,
\emph{mpmath: a Python library for arbitrary-precision floating-point
arithmetic} (version 1.3.0), \url{https://mpmath.org}
(accessed July 2026).

\bibitem{ivdoc}
The \texttt{mpmath} documentation, version 1.3.0,
\emph{Interval arithmetic (\texttt{mpmath.iv})},
\url{https://mpmath.org/doc/1.3.0/contexts.html}
(outward rounding of interval endpoints; accessed July 2026).

\bibitem{DS}
F.~Diamond and J.~Shurman,
\emph{A First Course in Modular Forms},
Grad.\ Texts in Math.\ \textbf{228}, Springer, New York, 2005.

\bibitem{MoKC}
R.~E. Moore, R.~B. Kearfott, and M.~J. Cloud,
\emph{Introduction to Interval Analysis},
SIAM, Philadelphia, 2009.

\end{thebibliography}
\end{document}